\documentclass[11pt, a4paper]{amsart}

\usepackage{amsfonts,amsmath,amssymb,amscd}

\newtheorem{theorem}{Theorem}[section]
\newtheorem{lemma}[theorem]{Lemma}
\newtheorem{prop}[theorem]{Proposition}
\newtheorem{corollary}[theorem]{Corollary}

\theoremstyle{definition}
\newtheorem{definition}[theorem]{Definition}
\newtheorem{rem}[theorem]{Remark}

\usepackage{color}

\newcommand\pf{\begin{proof}}
\newcommand\epf{\end{proof}}

\newcommand{\cmdblackltimes}{\mathop{\raisebox{0.2ex}{\makebox[0.92em][l]{${\scriptstyle\blacktriangleright\mathrel{\mkern-4mu}<}$}}}}

\numberwithin{equation}{section}

\title[Algebraic affine supergroup schemes]
{Harish-Chandra pairs for algebraic affine supergroup schemes over an arbitrary field}

\author[Akira Masuoka]{Akira Masuoka}
\address{Akira Masuoka: 
Institute of Mathematics, 
University of Tsukuba, 
Ibaraki 305-8571, Japan}
\email{akira@math.tsukuba.ac.jp}

\begin{document}

\begin{abstract}
Over an arbitrary field of characteristic $\ne 2$, we define 
the notion of Harish-Chandra pairs, and prove that the category of those pairs is anti-equivalent
to the category of algebraic affine supergroup schemes. The result is applied to 
characterize some classes of affine supergroup schemes such as those which are
(a) simply connected, (b) unipotent or (c) linearly reductive in positive characteristic.
\end{abstract}

\maketitle



\section{Introduction}\label{sec:intro}

What is an algebraic group? From the functorial view-point, it is defined as an \emph{affine group scheme}
over a fixed field, say $\Bbbk$, to be a representable group-valued functor defined on 
the category of commutative algebras over $\Bbbk$; see \cite{DG}.  
If $\Bbbk$ is an algebraically closed field of characteristic zero, an 
algebraic affine group scheme is the same as a linear algebraic group. 
Every affine group scheme is represented uniquely by a commutative Hopf algebra. 
Since affine group schemes thus correspond precisely to commutative Hopf algebras, they can be studied 
by Hopf-algebraic methods; see Hochschild \cite{H}, Takeuchi \cite{T-0}, \cite{T-I}, \cite{T-III}.  
Especially, Takeuchi made substantial contributions, which are free of characteristic of $\Bbbk$, 
replacing Lie algebras in zero characteristic case with \emph{hyperalgebras}. 

The symmetric tensor category of vector spaces is generalized by the category of super-vector spaces, that is, 
$\mathbb{Z}_2$-graded vector spaces, which has the familiar tensor product and the
super-symmetry; see \eqref{symmetry}. Each object, such as ordinary, Hopf or Lie algebras, defined in the
former category is generalized by a super-object, such as (Hopf or Lie) superalgebras, defined in the
latter category.

What is an algebraic supergroup? It is now automatic to answer this question; one has only to replace 
commutative algebras with the corresponding super-object. 
An \emph{affine supergroup scheme} over a field $\Bbbk$ is thus defined to be a representable
group-valued functor defined on the category of super-commutative superalgebras over $\Bbbk$. Here and
in what follows we pose the natural assumption $\operatorname{char} \Bbbk \ne 2$. An affine supergroup scheme
$G$ is represented uniquely by a super-commutative
Hopf superalgebra, say $A$, and is written as $G = \operatorname{SSp} A$. 
$G$ is said to be \emph{algebraic} if $A$ is \emph{affine}, that is, finitely generated.  
See \cite{MZ}, \cite{Z1}, \cite{Z2} for recent characteristic-free studies of supergroup schemes. 

The notion of \emph{Harish-Chandra pairs} was first introduced by Kostant \cite{Kostant}; it is such a pair $(G, L)$
of a linear algebraic (or analytic) group $G$ and a finite-dimensional Lie superalgebra $L$ that satisfies
some conditions. Recently, Carmeli and Fioresi \cite{CF} proved that there is a natural category-equivalence between 
the Harish-Chandra pairs and the algebraic affine supergroup schemes (in our terms), under the assumption 
that $\Bbbk$ is an algebraically closed field of characteristic zero; in the analytic situation, the same result had been
proved by Kostant \cite{Kostant} (see also Koszul \cite{Koszul}) 
in the real case, and by Vishnyakova \cite{V} in the real and complex cases both. 
In this paper we generalize the equivalence proved by Carmeli and Fioresi, removing the assumption on $\Bbbk$, and
apply the result for characteristic-free study of affine supergroup schemes. Just as shown by Takeuchi in the
non-super situation, super-cocommutative Hopf superalgebras, and especially irreducible ones 
(see Definition \ref{def:irreducible}) play an important role,
in place of Lie superalgebras. An irreducible super-cocommutative Hopf superalgebra is called a 
\emph{hyper-superalgebra}.  

Here we supply fundamental references, following kind suggestions by a referee. 
Manin \cite{Manin} defined for the first time superschemes in the purely algebraic setting, 
from a geometric view-point; they are defined as a special class of topological spaces
given structure sheaves of super-commutative superalgebras.  
Among other useful monographs we cite especially 
Berezin \cite{Berezin}, Varadarajan \cite{Varadarajan}, and Carmeli, Caston and Fioresi \cite{Carmeli-Caston-Fioresi}. 
Deligne and Morgan \cite{Deligne-Morgan}, and 
Carmeli et al. \cite{Carmeli-Caston-Fioresi} give modern treatments of the subject
from functorial and geometric view-points both, mainly in the differential-geometric setting.
Boseck \cite{Boseck1989}, \cite{Boseck1990} took the functorial (or Hopf-algebraic) view-point
in the purely algebraic setting, though the rational points considered were restricted to those in 
exterior algebras. 
The notion of hyper-superalgebras is not new. 
For example, the hyper-superalgebra of the algebraic supergroup $Q(n)$ was thoroughly 
studied by Brundan and Kleshchev \cite{Brundan-Kleshchev}.   

Let us describe the construction of this paper. 
After some preliminaries in Section 2, we will start with dualizing the notion of Harish-Chandra pairs. 
A \emph{dual Harish Chandra pair} is such a pair $(J, V)$ of a cocommutative Hopf algebra $J$ and 
a right $J$-module $V$, that is equipped with a bilinear map $[\hspace{2mm}, \ ] : V \times V \to P(J)$ with 
values in the Lie algebra $P(J)$ of all primitives in $J$, and that satisfies some conditions; see Definition 
\ref{def:DHCP}. 
Given a Lie superalgebra $L = L_0 \oplus L_1$, then the pair $(U(L_0), L_1)$, equipped with the original bracket
$[\hspace{2mm}, \ ] : L_1 \times L_1 \to L_0 \subset P(U(L_0))$ restricted to $L_1$, is a dual Harish-Chandra
pair provided $\operatorname{char} \Bbbk \ne 3$ (and $\operatorname{char} \Bbbk \ne 2$, as already assumed).  
If $\operatorname{char} \Bbbk = 0$, every dual Harish-Chandra pair $(J, V)$
with $J$ irreducible arises uniquely in this way from a Lie superalgebra. Given a dual Harish-Chandra pair $(J, V)$, we
construct a super-cocommutative Hopf superalgebra, $H(J, V)$; this equals the universal enveloping superalgebra
$U(L)$ if $(J, V)$ arises from a Lie superalgebra $L$ as above. We prove in Theorem \ref{thm:Takeuchi} that $(J, V)
\mapsto H(J, V)$ gives a category equivalence from the dual Harish-Chandra pairs $\mathsf{DHCP}$ to
the super-cocommutative Hopf superalgebras $\mathsf{CCHSA}$; this result was outlined 
by Takeuchi \cite{T2}. 

We will dualize rather faithfully the definition, the construction and the result above. 
Suppose that $C$ is an affine (that is, finitely generated commutative) Hopf algebra, and
$W$ is a finite-dimensional right $C$-comodule. Then the dual vector space $W^*$ of $W$ is naturally
a right module over the dual cocommutative Hopf algebra $C^{\circ}$ of $C$. The pair $(C, W)$ is called a
\emph{Harish-Chandra pair} if it is equipped with a bilinear map $W^* \times W^* \to P(C^{\circ})$ with
which $(C^{\circ}, W^*)$ forms a dual Harish-Chandra pair; see Definition \ref{def:HCP}. 
If $\Bbbk$ is an algebraically closed field of characteristic zero, every Harish-Chandra pair $(C, W)$ 
as just defined arises uniquely from a Harish-Chandra pair $(G, L)$ as defined by Carmeli and Fioresi, and others,
where $G = \operatorname{Sp} C(\Bbbk)$, the linear algebraic group of all rational points in $\Bbbk$, and 
$L = P(C^{\circ}) \oplus W^*$; see Remark \ref{rem:HCP}(2). 
Given a Harish-Chandra pair $(C, W)$ (in our sense), we
construct an affine Hopf superalgebra, $A(C, W)$; this construction of ours is different from the corresponding
construction given in \cite[Sect.~3.3]{CF}, and has an advantage in 
our verifying very easily that $A(C, W)$ is a Hopf superalgebra;
see the proof of Lemma \ref{lem:A(C,W)1}, and also Remark \ref{rem:PropHopf-struc}.
We prove in Theorem \ref{thm:equiv} that $(C, W)
\mapsto A(C, W)$ gives a category equivalence from the Harish-Chandra pairs $\mathsf{HCP}$ to
the affine Hopf superalgebras $\mathsf{AHSA}$; this last category is anti-isomorphic to
the category of algebraic affine supergroup schemes. 

Just as for group schemes, short exact sequences play an important role in theory of supergroup schemes.
Those sequences of affine supergroup schemes correspond precisely to those
sequences of super-commutative Hopf superalgebras. 
In Sections \ref{subsec:exact-sequence}, \ref{subsec:exact-seq-comm},
we reformulate, in a stronger form, some results from \cite[Sections 3, 5]{M} on short exact 
sequences of super-(co)commutative Hopf superalgebras, in terms of (dual) Harish-Chandra pairs.

In Sections \ref{sec:simply-conn}--\ref{sec:Nagata} the results obtained so far will be applied 
to characterize three classes of affine 
supergroup schemes. 
Theorem \ref{thm:simply-conn-zero} (resp., Theorem \ref{thm:simply-conn-p}) characterizes 
simply connected affine supergroup schemes in characteristic zero (resp., in positive characteristic)
in terms of the corresponding Lie superalgebras (resp., hyper-superalgebras); this directly generalizes
the corresponding result by Hochschild \cite{H} (resp., by Takeuchi \cite{T-III}) in the non-super situation. 
Theorem \ref{thm:unipotency} states that an affine supergroup scheme $G$ is unipotent if and only if the 
affine group scheme $G_{\operatorname{res}}$ obtained from $G$ by restricting its domain to the category of 
commutative algebras 
is unipotent; this is a very recent unpublished result by A.~N.~Zubkov, and we will give it a very simple proof.
Theorem \ref{thm:Nagata} shows that if $\operatorname{char} \Bbbk > 2$, 
every linearly reductive affine supergroup scheme
$G = \operatorname{SSp} A$ is necessarily purely even in the sense that $A$ consists of even elements only,
whence by Nagata's Theorem, it is of multiplicative type under the additional assumption that $G$ is algebraic and connected.   

We will work with a number of categories. Here is a table of their symbols:

\begin{alignat*}{2}
&\mathsf{CCHSA} & \qquad & \text{super-cocommutative Hopf superalgebras; \eqref{CCHSA-HySA}} \\
&\mathsf{HySA}    & \qquad & \text{hyper-superalgebras; \eqref{CCHSA-HySA}} \\
&\mathsf{LSA}      & \qquad & \text{Lie superalgebras; \eqref{CCHSA-HySA}} \\
&\mathsf{AHSA}      & \qquad & \text{affine Hopf superalgebras; Definition \ref{def:affine}} \\
&\mathsf{DHCP}    & \qquad & \text{dual Harish-Chandra pairs; Definition \ref{def:DHCP}}  \\
&\mathsf{iDHCP}   & \qquad & \text{irreducible dual Harish-Chandra pairs; Definition \ref{def:DHCP}} \\
&\mathsf{HCP}      & \qquad & \text{Harish-Chandra pairs; Definition \ref{def:HCP}} \\
&\mathsf{cHCP}    & \qquad & \text{connected Harish-Chandra pairs; Definition \ref{def:HCP}}
\end{alignat*} 

\section{Preliminaries}\label{sec:prelim}

\subsection{}\label{subsec:super-vec}
Throughout we work over a fixed field $\Bbbk$ whose characteristic $\operatorname{char} \Bbbk \ne 2$.
In particular the unadorned $\otimes$ denotes the tensor product over $\Bbbk$. 

A \emph{super-vector space} is a vector space $V = V_0 \oplus V_1$ graded 
by $\mathbb{Z}_2 = \{ 0, 1 \}$. Given a homogeneous element $v \in V$,
we let $|v| \in \mathbb{Z}_2$ denote its parity. The super-vector spaces, $V, W, \dots$, form a symmetric 
tensor category with respect to the familiar tensor product $V \otimes W$, and the super-symmetry
\begin{equation}\label{symmetry} 
c_{V,W} : V \otimes W \overset{\simeq}{\longrightarrow} W \otimes V,\quad 
c_{V,W}(v \otimes w) = (-1)^{|v||w|} w \otimes v.   
\end{equation}
We call objects, such as algebra or Hopf-algebra objects,
in this symmetric tensor category, attaching $`$super' to their original names,
so as \emph{superalgebras} or \emph{Hopf superalgebras}. They are said to be \emph{purely even} 
(resp., \emph{purely odd}), if 
they consist of even (resp., odd) elements only. 
To distinguish the tensor products of super(co)algebras, on which the super-symmetry does effect,
from those of non-super ones, we will write $A~\underline{\otimes}~B$ for $A \otimes B$. 

Given a vector space $V$, we let
$V^*$ denote the dual vector space. This is a super-vector space so that $(V^*)_i = (V_i)^*$, $i = 0,1$,  if $V$ is. 
By a \emph{pairing} of super-vector spaces $V$, $W$, we mean a bilinear form $\langle \hspace{2mm}, \ \rangle : V \times W 
\to \Bbbk$ such that $\langle V_i, W_j \rangle = 0$ if $i \ne j$ in $\mathbb{Z}_2$. A pairing induces 
linear maps $V \to W^*$, $W \to V^*$ preserving the parity. 

Given two pairings $\langle \hspace{2mm}, \ \rangle : V \times W \to \Bbbk$, $\langle \hspace{2mm}, \ 
\rangle : Z \times U \to \Bbbk$ of super-vector spaces, we define their tensor product 
$\langle \hspace{2mm}, \ \rangle : V \otimes Z \times W \otimes U \to \Bbbk$ by
$$ \langle v \otimes z,\ w \otimes u\rangle = \langle v, w \rangle\, \langle z, u\rangle, $$
where  $v \in V, w \in W, z \in Z, u \in U$.

\begin{lemma}\label{lem:tensored-pairing}
With the notation as above we have
$$\langle c_{Z,V}(z \otimes v), \ w \otimes u\rangle = \langle z\otimes v, \ c_{U,W}(u \otimes w)\rangle .$$
\end{lemma} 
\begin{proof}
We may suppose $|v| = |w|$, $|z| = |u|$, since otherwise, the both sides are equal to zero.
We see then that the both sides are equal to 
$-\langle v, w \rangle\, \langle z, u\rangle$ if $|v| = |z| = 1$, and to $\langle v, w \rangle\, \langle z, u\rangle$ 
in the remaining cases. 
\end{proof}

\subsection{}\label{subsec:Hopf-pair}
Given a coalgebra $C$, we let $\Delta : C \to C \otimes C$, $\varepsilon : C \to \Bbbk$ denote  
the coproduct and the counit, respectively. To present 
the coproduct explicitly, we use the Heyneman-Sweedler notation \cite[Sect.~1.2]{Sw} of the form
$$ \Delta(c) = \sum c_{(1)} \otimes c_{(2)},\ c \in C. $$ 

For a Hopf algebra or superalgebra $A$, $S : A \to A$ denotes the antipode, in addition to
$\Delta : A \to A \otimes A$, $\varepsilon : A \to \Bbbk$ as above. We let
$$ A^+ = \operatorname{Ker} \varepsilon $$
denote the augmentation ideal of $A$. 

A pairing $\langle \hspace{2mm}, \ \rangle : H \times A \to \Bbbk$ of Hopf superalgebras $H, A$ is called a \emph{Hopf pairing},
if we have
\begin{align}\label{Hopf-pairing}
 &\langle xy, a \rangle = \sum \langle x, a_{(1)}\rangle\, \langle y, a_{(2)}\rangle, \notag \\
 &\langle x, ab \rangle = \sum \langle x_{(1)}, a\rangle\, \langle x_{(2)}, b\rangle, \\
 &\langle 1, a \rangle = \varepsilon (a), \quad\langle x, 1 \rangle = \varepsilon (x), \notag
\end{align}
where $x, y \in H$, $a, b \in A$. It then results that
\begin{equation}\label{antipode}
\langle S(x), a \rangle  = \langle x, S(a) \rangle, \quad x\in H,\ a \in A. 
\end{equation}

\begin{definition}\label{def:N-graded}
We say that a Hopf superalgebra $A$ is \emph{$\mathbb{N}$-graded}, where 
$\mathbb{N} =\{ 0, 1, 2, \dots \}$ is the semigroup
of all non-negative integers, if $A$ is $\mathbb{N}$-graded, $A = \bigoplus_{n=0}^{\infty}\, A(n)$, as an
algebra and coalgebra, and if the $\mathbb{N}$-grading gives rise to the original $\mathbb{Z}_2$-grading so that
$$ A_0 = \bigoplus_{i\ge 0}A(2i), \quad  A_1 = \bigoplus_{i\ge 0}A(2i+1). $$
We say that a Hopf pairing $\langle \hspace{2mm}, \ \rangle : H \times A \to \Bbbk$ of $\mathbb{N}$-graded Hopf 
superalgebras $H, A$ is $\mathbb{N}$-\emph{homogeneous}, provided 
\begin{equation}\label{homogeneous}
\langle H(n), A(m) \rangle = 0 ~~ 
\text{if} ~~ n \ne m ~~ \text{in} ~~ \mathbb{N}.
\end{equation}
\end{definition}

A typical example of $\mathbb{N}$-graded Hopf superalgebras is the exterior algebra $\wedge(V)$ of
a vector space $V$, which is given the canonical $\mathbb{N}$-grading, and in which 
every element  $v$ in $V$ is \emph{primitive} \cite[p.199]{Sw}, that is, 
$\Delta(v) = 1 \otimes v + v \otimes 1$. 
Suppose $\dim V < \infty$. Then the canonical pairing $\langle \hspace{2mm}, \ \rangle : V \times V^* \to \Bbbk$ 
extends uniquely to an $\mathbb{N}$-homogeneous Hopf pairing 
$\langle \hspace{2mm}, \ \rangle : \wedge(V) \times \wedge(V^*) \to \Bbbk$,  
which is determined by
\begin{equation}\label{cano-pairing}
\langle v_1 \wedge \dots \wedge v_n, \, w_1 \wedge \dots \wedge w_n \rangle = 
\sum_{ \sigma \in \mathfrak{S}_n }\operatorname{sgn} \sigma \, \langle v_1, w_{ \sigma (1) } 
\rangle \dots \langle v_n, w_{ \sigma (n) } \rangle, 
\end{equation}
where $v_i \in V,\ w_i \in V^*,\ n > 0$. 
Note that this is a non-degenerate pairing; see Remark \ref{rem:sign-rule} below. 

\subsection{}\label{subsec:complete-Hopf} 
Suppose that $A = \bigoplus_{n=0}^{\infty}A(n)$ is an $\mathbb{N}$-graded Hopf 
superalgebra. Set 
\begin{equation}\label{hatA}
\widehat{A}=\prod_{n=0}^{\infty}A(n). 
\end{equation}
As a superalgebra this is the completion of $A$ with respect to the linear topology defined by the 
descending chain $I_n := \bigoplus_{i\ge n}A(i),\ n = 0, 1, \dots$, of super-ideals. 
The \emph{complete tensor product}  
$\widehat{A} ~ \widehat{\otimes} ~ \widehat{A}$ is the completion of the tensor product 
$\widehat{A} ~ \underline{\otimes} ~ \widehat{A}$ of superalgebras with respect the linear topology defined by 
the descending chain $\widehat{I}_n \otimes \widehat{A} + \widehat{A} \otimes \widehat{I}_n$, $n = 0, 1, \dots$, 
of super-ideals, where we set 
$\widehat{I}_n = \prod_{i\ge n}A(i)$. 
See \cite[Sect.~1.5]{T1} for the definition in a more general situation.   
Regard $\Bbbk$ as to be discrete. The structure maps on $A$ are completed to 
$\widehat{\Delta} : \widehat{A} \to \widehat{A} ~ \widehat{\otimes} ~ \widehat{A}$,
$\widehat{\varepsilon} : \widehat{A} \to \Bbbk$, $\widehat{S} : \widehat{A} \to
\widehat{A}$, which together satisfy the familiar Hopf-algebra axioms with $\otimes$ replaced 
by $\widehat{\otimes}$. Therefore, this may be called a \emph{complete topological Hopf superalgebra}. 
One sees that $A$ recovers from $\widehat{A}$ as
\begin{equation}\label{grhatA}
\operatorname{gr} \widehat{A} := \bigoplus_{n=0}^{\infty} \widehat{I}_n/\widehat{I}_{n+1}. 
\end{equation}
Note that an $\mathbb{N}$-homogeneous pairing $\langle \hspace{2mm}, \ \rangle : H \times A \to \Bbbk$ of
$\mathbb{N}$-graded Hopf superalgebras $H$, $A$ extends uniquely to a pairing
\begin{equation}\label{pairing-complete-general} 
\langle \hspace{2mm}, \ \rangle : H \times \widehat{A} \to \Bbbk
\end{equation}
such that for each $x \in H$, 
$\langle x, - \rangle : \widehat{A} \to \Bbbk$ is continuous.  We see that the extended pairing has the properties
which are the same as, or analogous to \eqref{Hopf-pairing}, \eqref{antipode}, \eqref{homogeneous}.

\subsection{}\label{subsec:dual-coalgebra}
Let $A$ be an algebra. Let $A^{\circ}$ denote the \emph{dual coalgebra} of $A$; see \cite[p.109]{Sw}.
By definition it consists of those elements in $A^*$ which annihilate some ideal $I \subset A$
of cofinite dimension, that is, $\dim A/I < \infty$. It follows that $A^{\circ}$ is the directed 
union $\bigcup_I (A/I)^*$ of the finite-dimensional coalgebras $(A/I)^*$. 

\begin{lemma}\label{lem:dual-super-coalgebra}
Suppose that $A$ is a superalgebra. Then $A^{\circ}$ consists of those elements in $A^*$ which 
annihilate some super-ideal $I \subset A$ of cofinite dimension. Therefore, $A^{\circ}$ is a super-coalgebra;
cf.~\cite[p.290]{M}.  
\end{lemma}
\begin{proof}
By \cite[Proposition~6.0.3]{Sw}, $A^{\circ}$ coincides with the pullback of $A^* \otimes A^*$ along the dual
map $A^* \to (A \otimes A)^*$ of the product on $A$. Since $A^* \otimes A^*$ is a super-vector subspace
of $(A \otimes A)^*$, it follows that $A^{\circ}$ is a super-vector subspace of $A^*$. 
The lemma follows, since 
if a homogeneous element in $A^*$ annihilates some ideal $I$ of $A$, then it annihilates the smallest 
super-ideal of $A$ including $I$. 
\end{proof}

\begin{corollary}\label{cor:dual-Hopf-superalgebra}
Suppose that $A$ is a Hopf superalgebra. 
Equipped with the ordinary dual algebra and coalgebra structures, 
$A^{\circ}$ forms a Hopf superalgebra. 
If $A$ is super-commutative (resp., super-cocommutative),
then $A^{\circ}$ is super-cocommutative (resp., super-commutative),
\end{corollary}
\begin{proof}
We see from Lemma \ref{lem:dual-super-coalgebra} that 
$A \mapsto A^{\circ}$ gives a contravariant functor from the category of superalgebras
to the category of super-coalgebras. 
Given superalgebras $A$, $B$, define  
$\varphi : A^{\circ}~\underline{\otimes}~B^{\circ} \to (A~\underline{\otimes}~B)^{\circ}$
by 
$$\varphi(f \otimes g)(a \otimes b) = f(a)g(b), \quad f \in A^{\circ}, g \in B^{\circ}, a\in A, b\in B.$$ 
Slightly modifying the proof of \cite[Lemma 6.0.1~b)]{Sw}, one sees that $\varphi$ is a linear isomorphism, by which
we will identify so as 
$A^{\circ}\, \underline{\otimes}~A^{\circ} = (A~\underline{\otimes}~A)^{\circ}$. 
Lemma \ref{lem:tensored-pairing} implies $(c_{A,B})^{\circ} = c_{B^{\circ},A^{\circ}}$, from which 
we see that $\varphi$ is an isomorphism
of super-coalgebras, and it makes $A \mapsto A^{\circ}$ into a tensor functor preserving the super-symmetry. 
Just as proving
the corresponding result in the non-super situation (see \cite[pp.122--123]{Sw}), we see that if $A$ is a Hopf 
superalgebra, then $A^{\circ}$ forms a Hopf superalgebra with respect to 
$\Delta^{\circ} : A^{\circ} \, \underline{\otimes} ~ A^{\circ} = (A ~ \underline{\otimes} ~ A)^{\circ} \to A^{\circ}$, $\varepsilon ^{\circ} : \Bbbk = \Bbbk^{\circ} \to A^{\circ}$, $S^{\circ} : A^{\circ} \to A^{\circ}$. 
The last assertion is now easy to see. 
\end{proof}

\begin{rem}\label{rem:sign-rule} 
Looking at the equations \eqref{Hopf-pairing} which define Hopf pairings, one might have felt it strange that the 
super-symmetry is not involved. But, Corollary \ref{cor:dual-Hopf-superalgebra} justifies it. One sees now that 
a pairing 
$\langle \hspace{2mm}, \ \rangle : H \times A \to \Bbbk$ of Hopf superalgebras $H, A$ is a Hopf pairing if and
only if $x \mapsto \langle x, -\rangle$ (or $a \mapsto \langle -, a\rangle$) gives a Hopf superalgebra map
$H \to A^{\circ}$ (or $A \to H^{\circ}$). For example, the 
non-degenerate Hopf pairing \eqref{cano-pairing} induces an isomorphism
$\wedge(V^*) \overset{\simeq}{\longrightarrow} (\wedge(V))^*$ of $\mathbb{N}$-graded Hopf superalgebras,
if $\dim V < \infty$.
\end{rem}

\subsection{}\label{subsec:corad}
Let $C$ be a super-coalgebra, or in other words, a $\mathbb{Z}_2$-graded coalgebra. 
One can construct the smash (or semi-direct) coproduct $\mathbb{Z}_2 \cmdblackltimes C$; 
this is the coalgebra constructed 
on the vector space $\Bbbk \mathbb{Z}_2 \otimes C$ with respect to the structure
\begin{equation}\label{smash-coproduct1}
\Delta(i \otimes c) = \sum \, (i \otimes c_{(1)})\otimes ((|c_{(1)}| + i) \otimes c_{(2)}),\quad
\varepsilon (i \otimes c) = \varepsilon(c),
\end{equation}
where $i = 0, 1$ in $\mathbb{Z}_2$, and $c \in C$. Note that a $C$-super-comodule is precisely a 
$\mathbb{Z}_2 \cmdblackltimes C$-comodule. 

Let $\operatorname{Corad} C$ be the coradical, that is, the (direct) sum of all 
simple subcoalgebras of $C$. Since it is stable under the 
$\mathbb{Z}_2$-action which naturally corresponds to the $\mathbb{Z}_2$-grading, $\operatorname{Corad} C$
is a super-subcoalgebra, so that $\mathbb{Z}_2 \cmdblackltimes \operatorname{Corad} C$ is constructed.

\begin{lemma}\label{lem:corad}
We have
$$ \mathbb{Z}_2 \cmdblackltimes \operatorname{Corad} C = \operatorname{Corad}(\mathbb{Z}_2  \cmdblackltimes C).$$
\end{lemma}
\begin{proof}
Since $C$ is a directed union of finite-dimensional super-subcoalgebras, 
we may and do assume $\dim C < \infty$. 
Set $R = C^*$, the dual algebra of $C$. Let $\operatorname{Rad} R$ be the Jacobson radical of $R$. 
Then, $(\operatorname{Corad} C)^* = R/\operatorname{Rad} R$. 
On $R$,\ $\mathbb{Z}_2$ acts as algebra-automorphisms by transposing  
the $\mathbb{Z}_2$-action on $C$, and $\operatorname{Rad} R$ is stable under the action; 
there arise, therefore, the semi-direct products below. 
The desired result follows by dualizing the well-known equality
\begin{equation}\label{Rad}
\mathbb{Z}_2 \ltimes \operatorname{Rad} R = \operatorname{Rad}(\mathbb{Z}_2 \ltimes R)
\end{equation}
in $\mathbb{Z}_2 \ltimes R$. 
\end{proof}

\begin{definition}\label{def:irreducible}
A Hopf superalgebra $A$ is said to be \emph{irreducible}, if $\operatorname{Corad} A = \Bbbk$; see \cite[p.157]{Sw}. 
By Lemma \ref{lem:corad}, this is equivalent to saying that the simple $A$-super-comodules are exhausted 
by the purely even or odd, trivial $A$-comodule $\Bbbk$. 
\end{definition}

Just as in the non-super situation (see \cite[Sect.~9.1]{Sw}), one sees that given a Hopf superalgebra $A$,
the largest subcoalgebra $A^1$ of $A$ such that $\operatorname{Corad} A^1 = \Bbbk$ is an irreducible 
Hopf super-subalgebra of $A$. We call this $A^1$ the \emph{irreducible component} of $A$ containing 1. 

\begin{definition}\label{hyper-superalgebra}
A \emph{hyper-superalgebra} is an irreducible super-cocommutative Hopf superalgebra.
\end{definition}

This is the direct generalization of the notion of \emph{hyperalgebras} \cite{T-0}, which 
are defined to be irreducible cocommutative Hopf algebras. We let
\begin{equation}\label{CCHSA-HySA} 
\mathsf{CCHSA},\quad \mathsf{HySA}, \quad \mathsf{LSA}  
\end{equation} 
denote the category of super-cocommutative Hopf superalgebras, the full subcategory
consisting of all hyper-superalgebras, and the category of Lie superalgebras, respectively. 
Given $L \in \mathsf{LSA}$, the universal enveloping superalgebra $U(L)$ uniquely forms a Hopf 
superalgebra in which every element in $L$ is primitive. Generated by primitives $L$, 
this $U(L)$ is a hyper-superalgebra; see \cite[Exercise 2) on p.224]{Sw}. 
Kostant's Theorem states that if 
$\operatorname{char} \Bbbk = 0$, 
then $L \mapsto U(L)$ gives a category equivalence
\begin{equation}\label{Kostant}
\mathsf{LSA} \approx \mathsf{HySA}. 
\end{equation}

\subsection{}\label{subsec:supergroup}
The category of super-commutative superalgebras has $\underline{\otimes}$ as coproduct. 
A representable group-valued
functor $G$ defined on that category is called an \emph{affine supergroup scheme}, which is represented
uniquely by a super-commutative Hopf superalgebra, say $A$, and is denoted by 
$$G = \operatorname{SSp} A. $$ 
Therefore, the category of affine supergroup schemes is anti-isomorphic to the category of super-commutative
Hopf superalgebras. We say that $G$ is \emph{purely even}, if $A$ is so, that is, if $A = A_0$.  

A \emph{left (resp., right) rational supermodule} over an affine supergroup scheme $G = \operatorname{SSp} A$
is by definition a right (resp., left) $A$-super-comodule. 

\begin{definition}\label{def:unipotent-reductive}
Let $G = \operatorname{SSp} A$ be an affine supergroup scheme.

(1) $G$ is said to be \emph{unipotent} if the simple rational $G$-supermodules are exhausted by the purely even or odd, 
trivial $G$-module $\Bbbk$, or equivalently if $A$ is irreducible. 

(2) $G$ is said to be \emph{linearly reductive} if every rational $G$-module is semisimple. By Lemma \ref{lem:corad},
this is equivalent to saying that $A$ is cosemisimple, that is, $A = \operatorname{Corad} A$. 
\end{definition} 

An affine supergroup scheme $G = \operatorname{SSp} A$ is said to be \emph{algebraic} if $A$ is finitely generated. 

\begin{definition}\label{def:affine}
A Hopf superalgebra is said to be \emph{affine}, if it is super-commutative and finitely
generated. We denote by $\mathsf{AHSA}$ the category of affine Hopf superalgebras. 
\end{definition}

The category of algebraic affine supergroup schemes is anti-isomorphic to $\mathsf{AHSA}$. 

\section{Super-cocommutative Hopf superalgebras and dual Harish-Chandra pairs}\label{sec:DHCP}

\subsection{}\label{subsec:DHCP1} 
Given a Hopf superalgebra $A$,  we let
$$ P(A) = \{ u \in A \mid \Delta(u) = 1 \otimes u + u \otimes 1 \} $$
denote the super-vector subspace of $A$ consisting of all primitives; this forms a 
Lie superalgebra with respect to the super-commutator
\begin{equation}\label{super-commutator}
[u, v] := uv - (-1)^{|u||v|} vu,
\end{equation}
where $u, v$ are homogeneous elements in $P(A)$. This notation may and will soon be used for
ordinary Hopf algebras, too. 

Let $J$ be a cocommutative Hopf algebra. Then $P(J)$ is stable under the right adjoint $J$-action
$$ u \mapsto \sum S(a_{(1)})u a_{(2)},\quad u \in P(J),\ a \in J.$$
Let $V$ be a right $J$-module. We denote the $J$-action on $V$ by $v \triangleleft a$,
where $v \in V, a \in J$. 
 
\begin{definition}\label{def:DHCP}
$(J, V)$ is called a \emph{dual Harish-Chandra pair}, if it is equipped, as its structure, with a bilinear map 
$[ \hspace{2mm}, \ ] : V \times V \to P(J)$ with values in the Lie algebra of all primitives in $J$,
such that
\begin{itemize}
\item[(a)] $\sum [u \triangleleft a_{(1)}, v \triangleleft a_{(2)}] = \sum S(a_{(1)})[u, v]a_{(2)},$
\item[(b)] $[u, v] = [v, u],$ 
\item[(c)] $v \triangleleft [v, v] = 0$ 
\end{itemize}
for all $u, v \in V,\ a \in J$. Note that Condition (c), applied to $u + v + w$ and combined with (b), implies
\begin{itemize}
\item[(d)] $u \triangleleft [v, w] + v \triangleleft [w, u] + w \triangleleft [u, v] = 0,\quad u, v, w \in V$.
\end{itemize}
Conversely, (d) implies (c) provided $\operatorname{char} \Bbbk \ne 3$.

A dual Harish-Chandra pair $(J, V)$ is said to be \emph{irreducible}, if $J$ is irreducible; see 
Definition \ref{def:irreducible}.   
 
A \emph{morphism} $(J, V) \to (J', V')$ of dual Harish-Chandra pairs is a pair of a
Hopf algebra map $f : J \to J'$ and a linear map $g : V \to V'$ such that
$$ g(v\triangleleft a) = g(v) \triangleleft' f(a),\quad f([u, v]) = [g(u), g(v)]' $$
for all $u, v \in V, \ a \in J$. The dual Harish-Chandra pairs and their morphisms naturally form a
category $\mathsf{DHCP}$. We let $\mathsf{iDHCP}$ denote the 
full subcategory consisting of all irreducible dual Harish-Chandra pairs.
\end{definition} 

\begin{rem}\label{rem:DHCPzero-char} 
(1) Let $L$ be a Lie algebra. Suppose that $V$ is a right $L$-Lie module, or equivalently 
a right module over the universal enveloping algebra $U(L)$ of $L$. 
Given a bilinear map $[ \hspace{2mm}, \ ] : V \times V \to L$,
extend it, as well as the bracket on $L$, to $L \oplus V$, by defining $[v, a] = -[a, v] := v \triangleleft a$ for $a \in L,
v \in V$. Assume $\operatorname{char} \Bbbk \ne 3$, adding to the original assumption 
$\operatorname{char} \Bbbk \ne 2$. 
Then we see that $(U(L), V)$ together with the given $[ \hspace{2mm}, \ ]$ is a dual Harish-Chandra pair if and only
if $L \oplus V$ forms a Lie superalgebra with respect to the extended bracket, in which $L$ is even, and $V$ odd.  
Indeed, in the definition above, the equations in Conditions (a), (d) are the same as the Jacobi identity, 
under (b). Thus, every Lie superalgebra $L_0 \oplus L_1$ gives rise to an irreducible dual Harish-Chandra 
pair $(U(L_0), L_1)$. 

(2) Suppose $\operatorname{char} \Bbbk = 0$.  
By Kostant's Theorem \cite[Theorem 13.0.1]{Sw} (see also \eqref{Kostant}), every hyperalgebra (that is, irreducible
cocommutative Hopf algebra) $J$ is of the form $U(L)$, where $L = P(J)$. It follows from Part 1 above that 
every irreducible Harish-Chandra pair arises uniquely from a Lie superalgebra.  

(3) Suppose $\operatorname{char} \Bbbk = 0$. Let $J$ be a cocommutative Hopf algebra. Suppose that $J$ is
pointed \cite[p.157]{Sw}; this necessarily holds if $\Bbbk$ is algebraically closed. Set $L = P(J)$, and let $G$
denote the group of all grouplikes in $J$.  
Again by Kostant's Theorem \cite[Theorem 8.1.5]{Sw}, $J = \Bbbk G \ltimes U(L)$, the smash 
(or semi-direct) product of $U(L)$ by the group algebra $\Bbbk G$. In particular, $L$ is stable under the 
adjoint $G$-action $a \mapsto a^g := g^{-1}ag$, where $a \in L, g \in G$. 
Suppose that we have an irreducible dual Harish-Chandra pair $(U(L), V)$; it arises uniquely from a Lie superalgebra
structure on $L \oplus V$, as was just seen.  
We see that $(J, V)$ together with the restricted bracket $[ \hspace{2mm}, \ ] : V \times V \to L$ on 
the Lie superalgebra $L \oplus V$
turns into a dual Harish-Chandra pair if and only if $V$ is a right $\Bbbk G$-module
such that $[v, a] \triangleleft g = [v\triangleleft g,\ a^g]$, where $v \in V, a \in L, g \in G$, and 
$[ \hspace{2mm}, \ ] : V \times V \to L$ is $G$-equivariant, where $G$ acts on $L$ by the adjoint action. 
\end{rem} 

Remark \ref{rem:DHCPzero-char}(2) shows the following.

\begin{prop}\label{prop:iDHCP}
If $\operatorname{char} \Bbbk =0$, then $L \mapsto (U(L_0), L_1)$ gives a category equivalence
$\mathsf{LSA} \approx \mathsf{iDHCP}$. 
\end{prop}

Suppose $H \in \mathsf{CCHSA}$. We define $\underline{H}$, $V_H$ as in \cite[p.291]{M}, as follows:
$$ \underline{H} = \Delta^{-1}(H_0 \otimes H_0), \quad V_H =P(H)_1. $$
Thus, $V_H$ consists of all odd
primitives in $H$, and $\underline{H}$ is seen to be the largest ordinary subcoalgebra of $H$; it is indeed
a Hopf subalgebra. 
One sees that the adjoint action
$$ v \triangleleft a = \sum S(a_{(1)}) v a_{(2)}, \quad v \in V_H, a \in \underline{H} $$
defines a right $\underline{H}$-module structure on $V_H$. 
The super-commutator \eqref{super-commutator} restricted to $V_H$ defines a bilinear form 
$$ [ \hspace{2mm}, \ ] : V_H \times V_H \to P(\underline{H}) $$
with values in $P(\underline{H})$, since the purely even primitives in $H$ constitute $P(\underline{H})$. 

\begin{prop}\label{prop:functor-underline}
$(\underline{H}, V_H)$, given the bilinear map above, is a dual Harish-Chandra pair. This construction
is functorial, so that $H \mapsto (\underline{H}, V_H)$ gives a functor $\mathsf{CCHSA} \to \mathsf{DHCP}$. 
\end{prop}
\begin{proof}
Condition (c) is satisfied, since one sees that $[v, v] = 2v^2 \in P(\underline{H})$, whence $v \triangleleft
[v,v] = v(2 v^2) - (2 v^2)v = 0$. The remaining is easy to see. 
\end{proof}

\subsection{}\label{subsec:DHCP2} 
We wish to construct a quasi-inverse of the functor just obtained. 

Given a vector space $V$, 
let $T(V)= \bigoplus_{n=0}^{\infty}\, T^n(V)$ denote the tensor algebra on $V$; this is $\mathbb{N}$-graded. 
The following is a special form of 
a known result in the braided situation; see \cite[Definition 3.2.3]{AG}, for example.

\begin{lemma}\label{lem:T(V)1}
Let $V$ be a vector space as above. The $\mathbb{N}$-graded algebra $T(V)$ turns uniquely into an
$\mathbb{N}$-graded Hopf superalgebra in which every element of $V$ is primitive. 
This $T(V)$ is super-cocommutative.  
\end{lemma}

\begin{rem}\label{rem:DeltaT(V)}
Let us give an explicit description of the coproduct $\Delta : T(V)\to T(V)~\underline{\otimes}~T(V)$,
which is called the \emph{shuffle coproduct}. 
It is the sum of linear maps $\Delta_{n, i} : T^n(V) \to T^i(V) \otimes T^{n-i}(V)$, where $0 \le i \le n$. 
Suppose $i = 0$ or $n$. Then, $\Delta_{n, 0} : T^n(V) \mapsto \Bbbk \otimes T^n(V)$ and  
$\Delta_{n,n} : T^n(V) \mapsto T^n(V) \otimes \Bbbk$ are the canonical isomorphisms.  
Suppose $0 < i < n$, and let 
\begin{equation}\label{i-shuffle}
\mathfrak{S}_{n,i} = \{ \sigma \in \mathfrak{S}_n \mid \sigma(1)<\dots<\sigma(i),\ \sigma(i+1)< \dots< \sigma(n) \}
\end{equation}
denote the subset of the symmetric group $\mathfrak{S}_n$ of degree $n$ which consists of all $i$-shuffles. Then, 
$$\Delta_{n,i}(v_1 \otimes \dots \otimes v_n) = \sum_{\sigma\in\mathfrak{S}_{n,i}} \operatorname{sgn} \sigma \, 
(v_{\sigma(1)}\otimes\dots\otimes v_{\sigma(i)})\otimes(v_{\sigma(i+1)}\otimes\dots\otimes v_{\sigma(n)}). $$ 
\end{rem}

Let $J$ be a cocommutative Hopf algebra, and let $V$ be a right $J$-module. Then $T(V)$  
turns naturally into a right $J$-module, with respect the diagonal $J$-action given by
$$ 1 \triangleleft a := \varepsilon(a)1, \quad (v_1 \otimes \dots \otimes v_n) \triangleleft a 
:= \sum (v_1 \triangleleft a_{(1)}) \otimes \dots \otimes 
(v_n \triangleleft a_{(n)}), $$
where $a \in J$, $v_i \in V$, $1 \le i\le n$. 

\begin{lemma}\label{lem:T(V)2}
On $T(V)$, the product $T(V)~\underline{\otimes}~T(V) \to T(V)$, the unit $\Bbbk \to T(V)$ and the other structure
maps $\Delta : T(V) \to T(V)~\underline{\otimes}~T(V)$, $\varepsilon : T(V) \to \Bbbk$,  
$S : T(V) \to T(V)$ are all $J$-linear, where $\Bbbk$ is regarded the trivial right $J$-module via the
counit. This means that $T(V)$ is a Hopf-algebra object in the 
symmetric tensor category $\mathsf{SMod}_J$ of right $J$-supermodules. 
\end{lemma}
\begin{proof} 
We see directly that the product and the unit are $J$-linear. 
The coproduct $\Delta$
is $J$-linear, since it is an algebra map, and is $J$-linear, restricted to $T^0(V) \oplus T^1(V)$. Similarly,
we see that $\varepsilon$ and $S$ are $J$-linear. 
\end{proof}

Let $J$ be a cocommutative Hopf algebra, and let $V$ be a right $J$-module.
Since $T(V)$ is in particular an algebra object in $\mathsf{SMod}_J$ by Lemma
\ref{lem:T(V)2}, we can construct the algebra
\begin{equation}\label{mathcalH} 
\mathcal{H}(J, V) := J \ltimes T(V) 
\end{equation}
of smash (or semi-direct) product; see \cite[Sect.~7.2]{Sw}. This is the tensor product $J \otimes T(V)$ as
a vector space, and is an $\mathbb{N}$-graded algebra with $ \mathcal{H}(J, V)(n) := J \otimes T^n(V),\ 
n \in \mathbb{N}$. 
The cocommutativity of $J$ ensures the following, just as in the non-super situation. 

\begin{prop}\label{prop:mathcalH}
$ \mathcal{H}(J, V) $ turns uniquely into an $\mathbb{N}$-graded Hopf superalgebra which 
includes $J = J \otimes \Bbbk$, 
$T(V) = \Bbbk \otimes T(V)$ as Hopf super-subalgebras. This $\mathcal{H}(J, V)$ is super-cocommutative,
and the antipode is given by
$$ S( a \otimes x) = (1 \otimes S(x))(S(a) \otimes 1) = \sum S(a_{(1)}) \otimes (S(x) \triangleleft S(a_{(2)})), $$
where $a \in J$, $x \in T(V)$
\end{prop}

The following result was outlined by M.~Takeuchi \cite{T2}. 

\begin{theorem}\label{thm:Takeuchi}
(1) Given $(J, V) \in \mathsf{DHCP}$, let $I(J, V)$ be the two-sided ideal of $\mathcal{H}(J, V)$
generated by all $1 \otimes (uv + vu) - [u, v] \otimes 1$, where $u, v \in V$. Then this is a Hopf 
super-ideal. Define
\begin{equation}\label{H(J,V)} 
H(J, V) :=  \mathcal{H}(J, V)/I(J, V).
\end{equation}
Then this is a super-cocommutative Hopf superalgebra. This construction is functorial, so that we have the
functor
$$ H : \mathsf{DHCP} \to \mathsf{CCHSA},\ (J, V) \mapsto H(J, V). $$

(2) The functor $H$ above is an equivalence which has as a quasi-inverse the functor $H \mapsto
(\underline{H}, V_H)$ given by Proposition \ref{prop:functor-underline}.
\end{theorem} 
\begin{proof}
Here we only prove Part 1 and defer the proof of Part 2 to Section \ref{DHCP4}. 
Since $1 \otimes (uv + vu) - [u, v] \otimes 1$ is even primitive, it follows that 
$I(J, V)$ is a Hopf super-ideal. The functoriality is easy to see.  
\end{proof} 

\begin{rem}\label{rem:Hzero-char}
We see from Proposition \ref{prop:iDHCP} that 
if $\operatorname{char} \Bbbk =0$, the equivalence above, restricted to $\mathsf{iDHCP}$, can be identified with the
category equivalence $\mathsf{LSA} \approx \mathsf{HySA}$ given in \eqref{Kostant}. 
\end{rem}

\subsection{}\label{DHCP4} 
To prove the remaining Part 2 above, let $(J, V) \in \mathsf{DHCP}$. Then $H(J, V)$ is in particular an
algebra given a natural algebra map from $J$; such an algebra is called a $J$-\emph{ring} in \cite[p.195]{B}. 
It is naturally a left (and right) $J$-module. Given an element $v \in V$, we write simply $v$ for the
natural image of $1 \otimes v \in \mathcal{H}(J, V)$ in $H(J, V)$.  

\begin{lemma}\label{lem:J-free} 
$H(J, V)$ is free as a left $J$-module. In fact, if we choose arbitrarily a basis $X$ of $V$
given a total order, then the products 
$$ x_1x_2\dots x_n,\quad x_i \in X,\ x_1 < x_2 < \dots < x_n,\ n \ge 0  $$
in $H(J, V)$ constitute a $J$-free basis. 
\end{lemma}

\begin{proof} 
Suppose that $X$ is a totally ordered basis of $V$. 
As a $J$-ring, $H(J, V)$ is generated by $X$, and is defined by the relations
\begin{itemize}
\item[(i)] $x a = \sum a_{(1)}(x \triangleleft a_{(2)}),\quad x \in X,\ a \in J,$ 
\item[(ii)] $xy =  -yx + [x, y],\quad x, y \in X,\ x > y,$
\item[(iii)] $x^2 = \frac{1}{2} [x, x],\quad x \in X. $
\end{itemize}
Here, Condition (b) of Definition \ref{def:DHCP} allows us to suppose in (ii) that $x > y$. We regard the
relations (i)--(iii) as a reduction system \cite[p.180]{B}, 
$$x a \to \sum a_{(1)}(x \triangleleft a_{(2)}),\quad xy \to -yx + [x, y],\quad x^2 \to \frac{1}{2} [x, x]. $$ 
Here, we suppose that $x \triangleleft a_{(2)}$ is presented as a left linear combination of elements in $X$.
 
We wish to apply the opposite-sided version of \cite[Proposition 7.1]{B}. 
For this we first define an order $\preceq$ 
among words $A$, $B, \dots$ in $X$ of a fixed length, as follows: we let $A \preceq B$, if
$A$ is a permutation of $B$, and contains the same or a smaller 
number of mis-ordered pairs as or than $B$. Here, a \emph{mis-ordered pair} in $A = x_1 x_2\dots x_n$ 
is a pair $(x_i, x_j)$ such that $i < j$, $x_i > x_j$. Next, we
add to $X$ one element $*$, which is a symbol representing any element in $J$. 
(To be more precise, $*$ represents the factor $J$ 
in $J \oplus \bigoplus_{x \in X} \mathbb{Z}x$, just as in \cite[Proposition 7.1]{B}, $\kappa$ represents $k$ in $M$.)
Let $A = x_1x_2\dots x_n$ be a word in $X \cup \{* \}$.
To this $A$, we associate a sequence $s(A) = (\epsilon_1,\epsilon_2,
\dots, \epsilon_n)$ of the numbers $0, 1$,  
following the rule that $\epsilon_i = 0$ if $x_i = *$, and $\epsilon_i = 1$ if $x_i \in X$. 
Let $A'$ denote the word in $X$ obtained from $A$ by removing all $*$.  
Given two words $A$, $B$ in $X \cup \{*\}$, we let $A \le B$, if one of the following holds:
\begin{itemize}
\item 
$\operatorname{length} A < \operatorname{length} B$;
\item
$\operatorname{length} A = \operatorname{length} B$, and $s(A)$ precedes $s(B)$ in lexicographical order,
where we suppose $0 < 1$;
\item
$\operatorname{length} A = \operatorname{length} B$, $s(A)= s(B)$ and $A' \preceq B'$. 
\end{itemize}
Then we see that this $\leq$ defines an order among all words in $X \cup \{*\}$, and satisfies 
the following assumptions required for the result we wish to apply: it is a 
semigroup order which satisfies the DCC, and which is consistent with our reduction system
in the sense that every reduction changes a word to a sum of smaller words.  
(The order $\leqslant$ on $\langle X \cup \{\kappa\}\rangle$ given in
Proposition 7.1 of \cite{B} is assumed to be induced from a pre-order $\leqslant_0$
on $\langle X \rangle$, while our order 
is not induced from such a pre-order.  However, the assumption is not necessary, and the proposition, in fact the
opposite-sided version, can apply to our situation as is seen from its proof.) 
Therefore, if we see that the ambiguities which may occur when we reduce the words
\begin{itemize}
\item[(iv)]  $xya,\quad x \ge y\ \text{in}\ X,\ a \in J$, 
\item[(v)]   $xyz,\quad x \ge y \ge z\ \text{in}\ X$
\end{itemize}
are all resolvable, Lemma \ref{lem:J-free} follows 
since the products given above are precisely the \emph{irreducible} words, that is, 
the words in $X$ which do not contain any $xy$ with $x \ge y$ in $X$ as a subword. 
First, let $xya$ be a word from (iv) with $x > y$.  
This is reduced on one hand as
$$ xya \to \sum xa_{(1)}(y\triangleleft a_{(2)}) \to \sum a_{(1)}(x \triangleleft a_{(2)})(y \triangleleft a_{(3)}), $$
and on the other hand as
\begin{align*}
xya & \to -yxa + [x, y]a\\
 & \to - \sum a_{(1)}(y \triangleleft a_{(2)})(x \triangleleft a_{(3)})+ \sum a_{(1)} 
[x\triangleleft a_{(2)}, y\triangleleft a_{(3)}], 
\end{align*}
where in the last step, we have used that the Hopf algebra axioms for $J$ and Condition (a) of Definition 
\ref{def:DHCP} give
$$ [x, y]a = \sum a_{(1)}S(a_{(2)})[x,y]a_{(3)} = \sum a_{(1)}[x\triangleleft a_{(2)}, y\triangleleft a_{(3)}]. $$ 
The two results are seen 
to be further reduced to the same element, by using the cocommutativity of $J$. Next, let 
$xyz$ be a word from (ii)  with $x > y > z$.   Note $x[y,z] \to [y,z]x + x \triangleleft [y,z]$, since 
$[y, z]$ is primitive. Then one sees that $xyz$ is reduced on one hand as
\begin{align*}
xyz &\to -xzy + x[y,z] \to zxy - [x,z]y + [y,z]x + x \triangleleft [y,z] \\
& \to -zyx + z[x,y] - [x,z]y + [y,z]x + x \triangleleft [y,z] \\
& \to -zyx + [x,y]z + z\triangleleft [x,y]  - [x,z]y + [y,z]x + x \triangleleft [y,z],  
\end{align*}
and on the other hand as
\begin{align*}
xyz &\to -yxz + [x,y]z \to yzx - y[x,z] + [x,y]z \\
& \to -zyx +[y,z]x - [x,z]y - y \triangleleft[x,z]  + [x,y]z. 
\end{align*}
The two results coincide by Conditions (b), (d) of  Definition \ref{def:DHCP}.
We have thus seen that the ambiguities in reducing the words in (iv), (v) are resolvable when $x, y$ and $z$ are
distinct. We can similarly prove the resolvability in the remaining cases. 
\end{proof} 

We need the following result from \cite{M}.

\begin{prop}[\cite{M}, Theorem 3.6] \label{prop:M1}
Let $H \in \mathsf{CCHSA}$, and choose arbitrarily a basis $X$ of $V_H$ given a total order. 
Then the left $\underline{H}$-linear map $\phi_X : \underline{H} \otimes \wedge(V_H) \to H$ defined
by
$$ \phi_X(1 \otimes(x_1 \wedge x_2 \wedge \dots \wedge x_n)) = x_1 x_2 \dots x_n$$
on the $\underline{H}$-free basis
$$ 1 \otimes(x_1 \wedge x_2 \wedge \dots \wedge x_n), \quad x_i \in X,\ x_1 < x_2 < \dots < x_n,\ n \ge 0$$
is a unit-preserving isomorphism of super-coalgebras. 
\end{prop}

\emph{Proof of Part 2 of Theorem \ref{thm:Takeuchi}.}
We will see that the composites of the two functors are isomorphic to the identity functors. 

Let $(J, V) \in \mathsf{DHCP}$, and set $H = H(J, V)$. By Lemma \ref{lem:J-free}, $H$ naturally includes
$J, V$ so that $J \subset \underline{H}, V \subset V_H$. Choose such a totally ordered basis $X$ of $V_H$ that
extends some totally ordered basis of $V$. Lemma \ref{lem:J-free} shows that the composite $J \otimes \wedge(V)
\hookrightarrow \underline{H} \otimes \wedge(V_H) \overset{\simeq}{\longrightarrow} H$ of
the inclusion with the isomorphism $\phi_X$ given by Proposition \ref{prop:M1} 
is an isomorphism. It follows that $(J, V) = (\overline{H}, V_H)$, 
which is seen to be a coincidence as objects in $\mathsf{DHCP}$. 

Let $H \in \mathsf{CCHSA}$, and construct $H(\underline{H}, V_H)$. We see that the inclusions
$\underline{H} \hookrightarrow H,\ V_H \hookrightarrow H$ uniquely extends to a Hopf superalgebra
map $\mathcal{H}(\underline{H}, V_H) \to H$, which factors through $H(\underline{H}, V_H)$.  The resulting map, 
which we denote by 
\begin{equation}\label{alpha}
\alpha_H :  H(\underline{H}, V_H) \to H,  
\end{equation}                                    
is natural in $H$, as is easily seen.
Choose a totally ordered basis $X$ of $V_H$. Then the isomorphism $\phi_X : \underline{H} \otimes \wedge(V_H) 
\overset{\simeq}{\longrightarrow} H$ shows that the left $\underline{H}$-free basis $x_1 x_2\dots x_n$ of 
$H(\underline{H}, V_H)$ obtained in Lemma \ref{lem:J-free} is mapped by $\alpha_H$ to such a basis
of $H$. Therefore, $\alpha_H$ gives a natural isomorphism.  
\quad $\square$ 

\begin{rem}\label{rem:restrected-equivalence}
As is easily seen, the obtained equivalence restricts to a category equivalence between $\mathsf{iDHCP}$ and 
$\mathsf{HySA}$.    
\end{rem}

\section{Affine Hopf superalgebras and Harish-Chandra pairs}\label{sec:HCP}

\subsection{}\label{subsec:HCP1}

We dualize the construction of $\mathcal{H}(J, V)$ given in the preceding section.

Let $W$ be a vector space. Let $T_c(W)$ be the \emph{tensor coalgebra} on $W$. It is, as a vector space,
the tensor algebra $T(W)$ so that $T_c(W) = \bigoplus_{n=0}^{\infty}T^n(W)$, and has 
the following coalgebra structure:
the counit $\varepsilon : T_c(W) \to \Bbbk$ is the projection onto the zero-th component, and the coproduct
$\Delta : T_c(W) \to T_c(W) \otimes T_c(W)$ is defined by $\Delta(1) = 1 \otimes 1$, and in 
positive degree $n > 0$, by
\begin{align*}
\Delta &( v_1 \otimes \dots \otimes v_n) 
= 1 \otimes (v_1 \otimes \dots \otimes v_n) + \\
& \sum_{0<i<n} (v_1 \otimes \dots v_i)\otimes(v_{i+1} \otimes \dots v_n)  
+ (v_1 \otimes \dots \otimes v_n)\otimes 1. 
\end{align*}
This forms an $\mathbb{N}$-graded Hopf superalgebra 
with respect to the following \emph{shuffle product}: 
$$ (v_1 \otimes \dots \otimes v_i)(v_{i+1} \otimes \dots \otimes v_n) 
= \sum_{\tau^{-1} \in \mathfrak{S}_{n,i}} (\operatorname{sgn} \tau)v_{\tau(1)}\otimes \dots \otimes v_{\tau(n)}.$$
Here, we emphasize that the inverse $\tau^{-1}$ of the running index $\tau$ is required to be an $i$-shuffle. 
See \cite[Definition 3.2.10]{AG}. The original 1 in $T^0(W)$ still acts as a unit in $T_c(W)$, and
$T_c(W)$ is super-commutative.

Let $C$ be a commutative Hopf algebra, and let $G = \operatorname{Sp} C$ be the corresponding 
affine group scheme. Suppose that $W$ is a right $C$-comodule; this is equivalent to saying that
$W$ is a rational left $G$-module. We present the comodule structure on $W$, as  
$$w \mapsto \sum w_{(0)} \otimes w_{(1)},\ W \to W \otimes C;$$
see \cite[p.33]{Sw}. Note that each component, $T^n(W)$, of $T_c(W)$, being a tensor product of right
$C$-comodules, is naturally a right $C$-comodule,  
and so their direct sum $T_c(W)$ is, too. 
Explicitly, the $C$-comodule structure $T^n(W) \to T^n(W) \otimes C$ is given by $1\mapsto 1\otimes 1$ when $n=0$,
and by
$$ w_1 \otimes \dots \otimes w_n \mapsto 
\sum \big((w_1)_{(0)} \otimes \dots \otimes (w_n)_{(0)}\big)
\otimes (w_1)_{(1)}\dots (w_n)_{(1)}$$
when $n>0$.
We have the following result, which dualizes Lemma \ref{lem:T(V)2}, and is proved easily in the dual manner.  

\begin{lemma}\label{lem:Tc(W)} 
$T_c(W)$ is a commutative Hopf-algebra object in the symmetric tensor category $\mathsf{SMod}^C$ of right 
$C$-super-comodules. 
\end{lemma}

It follows that the smash coproduct constructs the coalgebra
\begin{equation}\label{smash-coproduct2}
\mathcal{A}(C, W) := C \cmdblackltimes T_c(W).
\end{equation}
This coalgebra is constructed on the vector space $C \otimes T_c(W)$
with respect to the structure
\begin{equation}\label{smash-coproduct}
 \begin{aligned}
 \Delta(c \otimes z) &= \sum (c_{(1)}\otimes(z_{(1)})_{(0)})\otimes((z_{(1)})_{(1)}c_{(2)})\otimes z_{(2)}),\\ 
 \varepsilon(c \otimes z) &=\varepsilon(c) \varepsilon(z),
 \end{aligned}
\end{equation}
where $c \in C$, $z \in T_c(W)$; this generalizes \eqref{smash-coproduct1}. 
Regard $\mathcal{A}(C, W)$ as the algebra of the tensor product of 
$C$ and $T_c(W)$. Proposition \ref{prop:mathcalH} is easily dualized as follows. 

\begin{prop}\label{prop:mathcalA}
$\mathcal{A}(C, W)$ is a super-commutative $\mathbb{N}$-graded Hopf superalgebra whose $n$-th 
component is $C \otimes T^n(W)$. The antipode is given by
$$ S( c \otimes z ) = \sum S(cz_{(1)}) \otimes S(z_{(0)}), \quad c \in C,\ z \in T_c(W). $$
\end{prop}

Note from Section \ref{subsec:complete-Hopf} that the $\mathbb{N}$-graded Hopf superalgebra 
$\mathcal{A}(C, W)$ has the completion, which we denote by
\begin{equation}\label{hatmathcalA}
\widehat{\mathcal{A}}(C, W)= \prod_{n=0}^{\infty} C \otimes T^n(W). 
\end{equation}

\subsection{}\label{subsec:HCP2}

Let $W$ be a finite-dimensional vector space, and set $V = W^*$. 
We have the two $\mathbb{N}$-graded Hopf superalgebras $T(V)$, $T_c(W)$. The following is easy to see.@

\begin{lemma}\label{lem:pair}
The canonical pairings $\langle \hspace{2mm}, \ \rangle : T^n(V) \times T^n(W) \to \Bbbk$, $n \in \mathbb{N}$
defined by
$$\langle v_1\otimes \dots \otimes v_n ,  w_1\otimes \dots \otimes w_n \rangle
= \langle v_1,  w_1 \rangle \dots \langle v_n,  w_n \rangle$$
give rise to an $\mathbb{N}$-homogeneous Hopf pairing 
$\langle \hspace{2mm}, \ \rangle : T(V) \times T_c(W) \to \Bbbk$, which
is obviously non-degenerate. 
\end{lemma}

Let $C$ be an affine Hopf algebra, that is, a finitely generated commutative Hopf algebra. 
Set $J = C^{\circ}$, the dual cocommutative Hopf algebra of $C$; see Section \ref{subsec:dual-coalgebra}.
Let $\langle \hspace{2mm}, \ \rangle : J \otimes C \to \Bbbk$ be the canonical pairing.
Since $C$ is \emph{proper} in the sense that the canonical map $C \to J^*$ is an injection (see
\cite[Theorem 6.1.3]{Sw}), we have the following.

\begin{lemma}\label{lem:Cproper}
$\langle \hspace{2mm}, \ \rangle : J \otimes C \to \Bbbk$ is a non-degenerate Hopf pairing. 
\end{lemma}

Suppose that $W$ is a right $C$-comodule. Then $V$ has the transposed left $C$-comodule structure,
which induces the right $J$-module structure determined by
\begin{equation}\label{J-mod-struc}
\langle v \triangleleft a, w \rangle = \sum \langle v, w_{(0)}\rangle \, \langle a, w_{(1)} \rangle, 
\quad v \in V,\ a \in J,\ w \in W. 
\end{equation}

We have now the two $\mathbb{N}$-graded Hopf superalgebras $\mathcal{H}(J, V)$, $\mathcal{A}(C, W)$. 
The following is verified directly.@

\begin{prop}\label{prop:pair}
The pairing $\langle \hspace{2mm}, \ \rangle : \mathcal{H}(J, V) \times \mathcal{A}(C, W) \to \Bbbk$ defined
by 
$$ \langle a \otimes x ,\ c \otimes z \rangle = \langle a,\ c \rangle\, \langle x, z \rangle, \quad 
a \in J,\ x \in T(V),\ c \in J,\ z \in T_c(W)$$
is a non-degenerate $\mathbb{N}$-homogeneous Hopf pairing. 
\end{prop}

As was seen in \eqref{pairing-complete-general}, the pairing above naturally extends to 
\begin{equation}\label{pairing-complete}
\langle \hspace{2mm}, \ \rangle : \mathcal{H}(J, V) \times \widehat{\mathcal{A}}(C, W) \to \Bbbk. 
\end{equation}
  
One sees that $C$ is naturally a left $J$-module under the action $a \rightharpoonup c$, where $a \in J$, $c \in C$, 
defined by
$$ \langle b, \ a \rightharpoonup c \rangle = \langle ba,\ c \rangle,\quad b \in J.  $$

In the following, $\operatorname{Hom}_J$ represents the vector space of left $J$-linear maps,
while $\operatorname{Hom}$ represents that of $\Bbbk$-linear maps. Note
$$ \operatorname{Hom}_J(\mathcal{H}(J, V), C)
 = \prod_{n=0}^{\infty}\operatorname{Hom}_J(J \otimes T^n(V), C). $$ 
Just as $\widehat{\mathcal{A}}(C, W)$, this can be regarded as the completion of the $\mathbb{N}$-graded vector space
$\bigoplus_{n=0}^{\infty}\operatorname{Hom}_J(J \otimes T^n(V), C)$.  
Therefore we can construct the complete tensor product
\begin{align*}
&\operatorname{Hom}_J(\mathcal{H}(J, V), C) ~ \widehat{\otimes} ~ \operatorname{Hom}_J(\mathcal{H}(J, V), C)\\
&= \prod_{n, m \ge0} \operatorname{Hom}_J(J \otimes T^n(V), C)\otimes \operatorname{Hom}_J(J \otimes T^m(V), C),
\end{align*}
which is naturally included in 
\begin{align*}
&\operatorname{Hom}(\mathcal{H}(J, V) \otimes \mathcal{H}(J, V), \ C \otimes C)\\
&= \prod_{n, m \ge 0} \operatorname{Hom}((J \otimes T^n(V))\otimes (J \otimes T^m(V)),\ C\otimes C).
\end{align*}

\begin{prop}\label{prop:Hopf-struc}
(1) The linear isomorphisms $ C \otimes T^n(W) \overset{\simeq}{\longrightarrow} 
\operatorname{Hom}_J(J \otimes T^n(V), C),\ n \ge 0$
given by
$$ c \otimes z\ \mapsto\ ( a \otimes x\ \mapsto \ \langle x, z \rangle \, a \rightharpoonup c), $$
amount to a linear homeomorphism, 
$$ \xi : \widehat{\mathcal{A}}(C, W) \overset{\simeq}{\longrightarrow} \operatorname{Hom}_J(\mathcal{H}(J, V), C). $$

(2) Let $f, g \in \operatorname{Hom}_J(\mathcal{H}(J, V), C),\ X \in \mathcal{H}(J, V)$.  The product, the unit,
the counit $\widehat{\varepsilon}$ and the antipode $\widehat{S}$ on $\widehat{\mathcal{A}}(C, W)$ are transferred 
onto $\operatorname{Hom}_J(\mathcal{H}(J, V), C)$ via the $\xi$ above so that
\begin{align*}
fg(X) &= \sum f(X_{(1)})g(X_{(2)}), \\ 
\xi(1)(X) &= \varepsilon (X)1,\\
\widehat{\varepsilon}(f) &= \varepsilon(f(1)),\\
\langle a,\ \widehat{S}(f)(X)\rangle &= \sum \langle S(a_{(1)}),\ f(a_{(2)}S(X)S(a_{(3)})) \rangle, \quad a \in J. 
\end{align*}

(3) Via $\xi$ and $\xi ~ \widehat{\otimes} ~ \xi$,  the coproduct on $\widehat{\mathcal{A}}(C, W)$ 
is translated to 
\begin{align*}
\widehat{\Delta} : \operatorname{Hom}_J(\mathcal{H}(J, V), C) \to 
&\operatorname{Hom}_J(\mathcal{H}(J, V), C) ~ \widehat{\otimes} ~ \operatorname{Hom}_J(\mathcal{H}(J, V), C)\\
&\subset \operatorname{Hom}(\mathcal{H}(J, V)\otimes \mathcal{H}(J, V), C\otimes C)
\end{align*}
that is determined by 
\begin{equation}\label{hatDelta} 
\langle a\otimes b,\ \widehat{\Delta}(f)(X \otimes Y) \rangle 
= \sum \langle ab_{(1)},\ f(S(b_{(2)}) X b_{(3)} Y) \rangle, 
\end{equation}
where $a, b \in J,\ f \in \operatorname{Hom}_J(\mathcal{H}(J, V), C),\ X, Y \in \mathcal{H}(J, V)$. 
\end{prop}
\begin{proof}
(1) This is easy to see.

(2) We verify the formula for $\widehat{S}$; the rest is verified similarly. We may suppose that $f$, 
restricted to $J \otimes T^n(V)$, is given by 
$$ f(b \otimes x) = \langle x, z \rangle \, b\rightharpoonup c, \quad b \in J,~~x \in T^n(V) $$ 
for some $c \in C$, $z \in T^n(W)$.  
Recall from Proposition \ref{prop:mathcalA} the description of the antipode of ${\mathcal{A}}(C, W)$. 
Then one sees
$$ \widehat{S}(f)(b \otimes x) = \sum \langle x, S(z_{(0)}) \rangle \, b\rightharpoonup S(cz_{(1)}). $$  
Therefore, the desired formula will follow if we see that for every $a \in J$,
\begin{align*} 
&\sum \langle ab, S(cz_{(1)}) \rangle\, \langle x, S(z_{(0)}) \rangle\\ 
&=  
\sum \langle S(a_{(1)}),\ a_{(2)}S(b_{(1)})S(a_{(3)}) \rightharpoonup c \rangle\, \langle S(x) \triangleleft 
S(b_{(2)}) S(a_{(4)}),\ z  \rangle. 
\end{align*}
Here one should recall from Proposition \ref {prop:mathcalH} the description of the antipode of 
$\mathcal{H}(J, V)$. Using the antipodal identity for $J$ and Equation \eqref{J-mod-struc}, 
we compute that the right-hand side equals
$$ \sum \langle S(b_{(1)})S(a_{(1)}), c \rangle\, \langle S(x), z_{(0)}  \rangle\, 
\langle S(b_{(2)})S(a_{(2)}), z_{(1)} \rangle, $$
which is seen to equal the left-hand side by Lemmas \ref{lem:pair}, \ref{lem:Cproper}. 

(3) We may suppose that $f$ is as above. Then the desired formula will follow if we see that 
\begin{align*}
&\sum \langle a, \ g \rightharpoonup c_{(1)} \rangle\ 
\langle x,\ (z_{(1)})_{(0)} \rangle \
\langle b,\ h \rightharpoonup (z_{(1)})_{(1)}c_{(2)} \rangle\
\langle y,\ z_{(2)} \rangle \\
&=\sum \langle ab_{(1)},\ S(b_{(2)})g b_{(3)} h_{(1)} \rightharpoonup c\rangle\, 
\langle (x \triangleleft b_{(4)}h_{(2)})y,\ z \rangle,
\end{align*}
where $a, b \in J$ and $g \otimes x, h \otimes y \in \mathcal{H}(J, V)$. 
We compute that the left-hand side equals
$$ \sum \langle ag,\  c_{(1)} \rangle\ 
\langle x\triangleleft b_{(1)}h_{(1)},\ z_{(1)} \rangle \
\langle b_{(2)}h_{(2)}, c_{(2)} \rangle\
\langle y,\ z_{(2)} \rangle,$$
which is seen to equal the right-hand side since $J$ is cocommutative.  
\end{proof}

\subsection{}\label{subsec:HCP3} 
Let $C$, $J$ be as in the preceding subsection. Thus, $C$ is an affine Hopf algebra, 
and $J = C^{\circ}$. 
Set $G = \operatorname{Sp} C$, the algebraic affine group scheme represented by $C$. By definition
the Lie algebra $P(J)$ of primitives in $J$ is the Lie algebra $\operatorname{Lie} G$ of $G$. 
Recall $C^+ = \operatorname{Ker} \varepsilon$. The quotient vector space $C^+/(C^+)^2$ 
of $C^+$ divided by its square is the cotangent space of $G$ at 1, and is dual to $P(J)$. 

Let $J^1$ denote the irreducible component  
of $J$ containing 1; see Section \ref{subsec:corad}. This consists
of those elements in $C^*$ which annihilate some power $(C^+)^n$ of $C^+$. 
This is called the \emph{hyperalgebra} of $G$, denoted
$\operatorname{hy} G = J^1$.  One sees $P(J) \subset J^1$. Summarizing we have
\begin{equation}\label{Lie} 
P(J^1) = P(J) = \operatorname{Lie} G = (C^+/(C^+)^2)^*.
\end{equation} 

\begin{rem}\label{rem:adjoint-coaction}
Consider the right adjoint coaction 
\begin{equation*}
c \mapsto \sum c_{(2)} \otimes S(c_{(1)})c_{(3)},\ C \to C \otimes C 
\end{equation*}
by $C$ on itself; this corresponds to the right adjoint action by $G$ on itself.
This induces on $C^+/(C^+)^2$ a right $C$-comodule structure, which in turn is transposed
to a left $C$-comodule structure on $P(J)$. This last induces on $P(J)$
structures of right modules over $J$, $J^1$ and $G$. One sees that they
coincide with the familiar adjoint actions. 
\end{rem} 

If $G$ is connected, or equivalently if the prime spectrum $\operatorname{Spec} C$ of $C$ is connected,
we say that $C$ is \emph{connected}.  The following is well-known; see \cite[Proposition 0.3.1(g)]{T-I}. 

\begin{lemma}\label{lem:conn}
Suppose that $C$ is connected. Then the canonical pairing $\langle \hspace{2mm}, \ \rangle : J^1 \times C 
\to \Bbbk$, which is restricted from $\langle \hspace{2mm}, \ \rangle : J \times C \to \Bbbk$, is non-degenerate. 
\end{lemma} 

\begin{rem}\label{rem:Kostant} 
(1) If $\operatorname{char} \Bbbk = 0$, then $J^1 = U(\operatorname{Lie} G)$; see Remark \ref{rem:DHCPzero-char}(2). 

(2) Assume that $\Bbbk$ is algebraically closed. Then by Kostant's Theorem \cite[Theorems 8.1.5]{Sw}, 
$J = \Bbbk G(\Bbbk) \ltimes J^1 $, the smash (or semi-direct) product of $J^1$ by the group algebra 
of the linear algebraic group $G(\Bbbk)$.  
\end{rem}

\begin{rem}\label{rem:PropHopf-struc}
Suppose that we are in the situation of Proposition \ref{prop:Hopf-struc}. We have the dual Harish-Chandra
pair $(J^1, V)$ with the restricted action on $V$ by $J^1$, which gives rise to the Hopf superalgebra
$\mathcal{H}(J^1, V)$; it is a Hopf super-subalgebra of $\mathcal{H}(J, V)$ such that $\mathcal{H}(J, V) =
J \otimes_{J^1} \mathcal{H}(J^1, V)$ as a left $J$-module. This last implies 
$$\operatorname{Hom}_J(\mathcal{H}(J, V), C) = \operatorname{Hom}_{J^1}(\mathcal{H}(J^1, V), C).$$
Therefore, in Proposition \ref{prop:pair} we may replace the former with the latter. Moreover, in the formulas for
$\widehat{S}$, $\widehat{\Delta}$, we may suppose by Lemma \ref{lem:conn} that $a, b \in J^1$, 
if $G$ is connected. We may suppose that $a, b \in G(\Bbbk)$, if $\Bbbk$ is an algebraically closed 
field of characteristic zero, since $C$ is then the 
algebra of all polynomial functions on the linear algebraic group $G(\Bbbk)$. 
In this last situation the resulting formulas are essentially the same as those given by 
Carmeli and Fioresi \cite[Proposition 3.10]{CF}. But, contrary to the method of \cite{CF} or \cite{V}, we
will not depend on the formulas to construct 
affine Hopf superalgebras (or algebraic affine supergroup schemes) 
from Harish-Chandra pairs. Especially, we cannot use the formula \eqref{hatDelta} to define their coproducts.
For we do not see that $\widehat{\Delta}(f)(X\otimes Y)$ has its value in $C \otimes C$, because we do not have such a
characterization as in the situation of \cite{CF} that $C \otimes C$ is the algebra of the polynomial functions 
on $G(\Bbbk) \times G(\Bbbk)$. 
\end{rem} 

\subsection{}\label{subsec:HCP4}
Let $C$ be an affine Hopf algebra, and let $W$ be a finite-dimensional right $C$-comodule. Set
$$ G = \operatorname{Sp} C,\quad J = C^{\circ},\quad V = W^*,$$ 
as in the preceding two subsections. Recall from \eqref{J-mod-struc} that $V$ is naturally a
right $J$-module. 

\begin{definition}\label{def:HCP}
$(C, W)$ is called a \emph{Harish-Chandra pair}, if it is equipped, as its structure, with 
a bilinear map $[ \hspace{2mm}, \ ] : V \times V \to P(J)$
with which $(J, V)$ is a dual Harish-Chandra pair; see Definition \ref{def:DHCP}. 
It is said to be \emph{connected}, if $C$ is connected. 

Suppose that $(C', W')$ is another Harish-Chandra pair with $J' = (C')^{\circ}$, $V' = (W')^*$. 
A morphism $(C, W) \to (C', W')$ of Harish-Chandra pairs is a pair of a Hopf algebra map $f : C \to C'$ and 
a linear map $g : W \to W'$ such that the dual Hopf algebra map $f^{\circ} : J' \to J$ of $f$ and the dual linear map
$g^* : V' \to V$ constitute a morphism $(J', V') \to (J, V)$ of dual Harish-Chandra pairs. 

The Harish-Chandra pairs and their morphisms 
naturally form a category $\mathsf{HCP}$. 
We let $\mathsf{cHCP}$ denote the full subcategory of 
$\mathsf{HCP}$ consisting of all connected Harish-Chandra pairs.

In the situation above, the object $(J, V)$ (resp., the morphism $(f^{\circ}, g^*)$) in $\mathsf{DHCP}$
is said to be \emph{associated with} the object $(C, W)$ (resp., the morphism $(f, g)$) in $\mathsf{HCP}$. 
\end{definition}

\begin{rem}\label{rem:HCP}
(1) In the situation above, Condition (a) of Definition \ref{def:DHCP} is equivalent to the condition that 
$[ \hspace{2mm}, \ ]$, regarded as a linear map $V \otimes V \to P(J)$, is left $C$-colinear; see Remark
\ref{rem:adjoint-coaction}. One sees from Lemma \ref{lem:conn} that if $C$ is connected, then  
the last condition is equivalent to that $[\hspace{2mm}, \ ]$ is left $J^1$-linear. It follows that if $C$ is connected,
then $(C, W)$ is a Harish-Chandra pair if the associated bilinear map makes $(J^1, V)$ into a dual Harish-Chandra
pair.

(2) Suppose $\operatorname{char} \Bbbk = 0$. We see from Proposition \ref{prop:iDHCP} and 
the first half of Part 1 above that $\mathsf{HCP}$ is anti-isomorphic to the category of those pairs $(G, L)$
of an algebraic affine group scheme $G$ and a finite-dimensional Lie superalgebra $L$, given a rational
right $G$-module structure on $L_1$, which satisfy
\begin{itemize}
\item[(i)] $L_0 = \operatorname{Lie} G$, 
\item[(ii)] the right adjoint action on $L_1$
by $L_0$ coincides with the action arising from the given rational right $G$-module structure, and
\item[(iii)] the bracket $[\hspace{2mm}, \ ] : L_1 \times L_1 \to L_0$ on $L$ restricted to $L_1$ is $G$-equivariant,
where $L_0$ is regarded as a rational right $G$-module by the adjoint action. 
\end{itemize}
In \cite{CF} and others, the latter category is defined as the category of (super) Harish-Chandra pairs.
More precisely, they assume in addition that $\Bbbk$ is algebraically closed, to identify rational $G$-modules
with rational modules over the linear algebraic group $G(\Bbbk)$. 
\end{rem}

\begin{prop}\label{prop:connHCP} 
Suppose $\operatorname{char} \Bbbk = 0$. Then 
$\mathsf{cHCP}$ is anti-isomorphic
to the category of the pairs $(G, L)$ of a connected algebraic affine group scheme $G$ and a
finite-dimensional Lie superalgebra $L$ such that $L_0 = \operatorname{Lie} G$, and the adjoint action on $L_1$
by $L_0$ arises (necessarily uniquely) from a rational $G$-module structure.
\end{prop}  
\begin{proof}
This follows from Lemma \ref{lem:conn} and the second half of Remark \ref{rem:HCP}(1). 
\end{proof}

In general we have the following, which follows from definitions and Lemmas \ref{lem:Cproper},
\ref{lem:conn}.

\begin{prop}\label{prop:dualization}
(1) $(C, W) \mapsto (C^{\circ}, W^*)$ defines a contravariant functor,
$$-^{\circ} : \mathsf{HCP} \to \mathsf{DHCP},$$
which is faithful.

(2) $(C, W) \mapsto ((C^{\circ})^1, W^*)$ defines a contravariant functor,
$$(-^{\circ})^1 : \mathsf{HCP} \to \mathsf{iDHCP},$$
which is faithful when it is restricted to $\mathsf{cHCP}$. 
\end{prop} 

\subsection{}\label{subsec:HCP5}
Suppose that $A$ is a super-commutative Hopf superalgebra. Let $I_A = AA_1$ be the super-ideal
of $A$ generated by the odd part $A_1$; this is the smallest super-ideal such that $A/I_A$
is an ordinary algebra. Recall $A^+ = \operatorname{Ker} \varepsilon$, and 
let $A_0^+ = A_0 \cap A^+$. We define $\overline{A}$, $W^A$, as in \cite{M}, by
$$ \overline{A} := A/I_A = A_0/A_1^2,\quad W^A := (A^+/(A^+)^2)_1 = A_1/A_0^+A_1.$$

One sees that $\overline{A} $ is indeed an ordinary quotient Hopf algebra of $A$. Let $G = \operatorname{SSp} A$.
Then $\overline{A}$ represents an affine group scheme,
\begin{equation}\label{Gres} 
G_{\operatorname{res}} = \operatorname{Sp} \overline{A},
\end{equation}
which is the group-valued functor obtained from $G$ by restricting its domain to the category of 
commutative algebras; see
\cite{MZ}. We call $G_{\operatorname{res}}$ the affine group scheme \emph{associated with} $G$. 

One sees that $W^A$ is the odd part of the cotangent space $A^+/(A^+)^2$ of $G$ at 1. It is known that 
$A$ is finitely generated if and only if $\overline{A}$ is finitely generated, $\dim W^A < \infty$ 
and the ideal $I_A$ of $A$ is nilpotent; see \cite[Proposition 4.4]{M}. One sees from \cite[Theorem 4.5]{M}
(see also Proposition \ref{prop:M2} below) that $A$ or $G$ is purely even if and only if $W^A = 0$. 

Let $A \in \mathsf{AHSA}$, and set $H = A^{\circ}$. Then, $H \in \mathsf{CCSHA}$. 
We see that the canonical pairing $\langle \hspace{2mm}, \ \rangle : H \times A \to \Bbbk$ is a Hopf pairing.
Since $\langle \underline{H}, A_1\rangle = 0$, a Hopf pairing $\underline{H} \times \overline{A} \to
\Bbbk$ is induced.
 
Recall from Section \ref{subsec:DHCP1} 
the definitions of $\underline{H}$, $V_H$. One sees from \cite[Proposition 4.3]{M}
$$V_H = (W^A)^*.$$
Therefore, we have the canonical Hopf pairing  
$\langle \hspace{2mm}, \ \rangle : \wedge(V_H) \times \wedge(W^A) \to \Bbbk$, 
as given by \eqref{cano-pairing}.

Choose a totally ordered basis $X$ of $V_H$. Recall from Proposition \ref{prop:M1} that we have the isomorphism
$\phi_X : \underline{H} \otimes \wedge(V_H) \overset{\simeq}{\longrightarrow} H$. 

\begin{prop}\label{prop:M2}
There exists uniquely a counit-preserving, left $\overline{A}$-colinear isomorphism 
$\psi_X : A \overset{\simeq}{\longrightarrow} \overline{A} \otimes \wedge(W^A)$ of superalgebras,
such that 
\begin{equation}\label{phi-psi}
\langle \phi_X(h \otimes u), \ a \rangle = \langle h \otimes u, \ \psi_X(a) \rangle, \quad h \in \underline{H}, u \in
\wedge(V_H), a \in A.
\end{equation}
\end{prop} 
\begin{proof}
Define a unit-preserving super-coalgebra map $\iota_X : \wedge(V_H) \to H$ by
\begin{equation}\label{iota}
\iota_X(x_1 \wedge \dots \wedge x_n) = x_1 \dots x_n, 
\end{equation} 
where $x_i \in X,\ x_1 <\dots < x_n,\ 0 \le n \le \#X$. 
Then we have uniquely a counit-preserving superalgebra map $\pi_X : A \to \wedge(W^A)$ that satisfies
\begin{equation}\label{pi}  
\langle \iota_X(u),\ a \rangle =  \langle u, \ \pi_X(a) \rangle, \quad u \in \wedge(V_H),\ a \in A. 
\end{equation}
Define $\psi_X : A \to \overline{A} \otimes \wedge(W^A)$ by 
\begin{equation}\label{psi}
\psi_X(a) = \sum \overline{a}_{(1)} \otimes \pi_X(a_{(2)}).
\end{equation}
Here and in what follows, $a \mapsto \overline{a}$ denotes the quotient map $A \to \overline{A}$. 
It is shown in the proof of \cite[Theorem 4.5]{M} that $\psi_X$ is an isomorphism. One sees easily that it
has the desired properties. 
\end{proof} 
 
\begin{corollary}\label{cor:non-degenerate}
We have $\underline{H} = \overline{A}^{\circ}$, and  
the canonical pairing $\langle \hspace{2mm}, \ \rangle : H \times A \to \Bbbk$ is non-degenerate. 
\end{corollary} 
\begin{proof}
This corollary follows immediately from Proposition \ref{prop:M2}.
\end{proof}

\begin{lemma}\label{Wcomod} 
Keep the situation as above.
The right adjoint $\overline{A}$-coaction on $A$ 
$$ a \mapsto \sum a_{(2)} \otimes S(\overline{a}_{(1)})\overline{a}_{(3)},\ A \to A \otimes \overline{A} $$
induces on $W^A$ a right $\overline{A}$-comodule structure. 
\end{lemma}
\begin{proof}
This is seen since the structure map $A \to A \otimes \overline{A}$ is a superalgebra map which is compatible
with the counit $\varepsilon : A \to \Bbbk$. 
\end{proof}

The induced coaction as well will be called the \emph{adjoint coaction}. 
We will regard $W^A$ as a right $\overline{A}$-comodule by this coaction. 
It is transposed to a left $\overline{A}$-comodule structure on $V_H = (W^A)^*$.

Keep the situation as above. We obtain $(\underline{H}, V_H) \in \mathsf{DHCP}$
from $H = A^{\circ}$; see Proposition \ref{prop:functor-underline}. 
Similarly to Remark \ref{rem:adjoint-coaction}, we see that the right module structure 
by $\underline{H} = \overline{A}^{\circ}$ on $V_H$ which is 
induced from the transposed left $\overline{A}$-comodule structure 
coincides with the right adjoint $\underline{H}$-action. This proves the first half of the following. 

\begin{prop}\label{prop:functor-overline}
$( \overline{A}, W^A )$, given the bilinear map $[ \hspace{2mm}, \ ] : V_H \times V_H \to P(\underline{H})$
associated with $(\underline{H}, V_H) \in \mathsf{DHCP}$, is a Harish-Chandra pair. This construction
is functorial, so that $A \mapsto (\overline{A}, W^A)$ gives a functor $\mathsf{AHSA} \to \mathsf{HCP}$. 
\end{prop}
\begin{proof}
A morphism $f : A \to B$ in $\mathsf{AHSA}$ induces naturally a Hopf algebra map 
$\overline{f} : \overline{A} \to \overline{B}$ and a linear map $W^f : W^A \to W^B$. 
It follows by the functoriality shown in \cite[Remark 4.8]{M} and in Proposition 
\ref{prop:functor-underline} that the two maps
form a morphism in $\mathsf{DHCP}$, since their duals $\overline{f}^{\circ}$, $(W^f)^*$ coincide with 
the maps which arise from the morphism $f^{\circ} : B^{\circ} \to A^{\circ}$ in
$\mathsf{CCHSA}$. This implies the desired functoriality.   
\end{proof}

\begin{rem}\label{grA}
For later use, recall from \cite{M} that given $A \in \mathsf{AHSA}$, we have the  
$\mathbb{N}$-graded Hopf superalgebra
$$\operatorname{gr} A := \bigoplus_{n=0}^{\infty} I_A^n/I_A^{n+1}.$$
The adjoint $\overline{A}$-coaction on $W^A$ extends uniquely to $\wedge(W^A) \to \wedge(W^A) \otimes
\overline{A}$ with which $\wedge(W^A)$ is a Hopf-algebra object in $\mathsf{SMod}^{\overline{A}}$; see
Lemma \ref{lem:Tc(W)}.
Just as in \eqref{smash-coproduct},   
the smash coproduct constructs an $\mathbb{N}$-graded Hopf superalgebra, 
$\overline{A}\cmdblackltimes \wedge(W^A)$,  with the algebra structure of tensor product.  
By \cite[Proposition 4.9(2)]{M}, there exists uniquely a natural isomorphism
$\psi : \operatorname{gr} A \overset{\simeq}{\longrightarrow} \overline{A}\cmdblackltimes \wedge(W^A)$
such that $\psi(0) : \operatorname{gr} A(0) = \overline{A} \to \overline{A}$ is the identity on $\overline{A}$, and 
$\psi(1) : \operatorname{gr} A(1) = A_1/A_1^3 \to \overline{A} \otimes W^A$ induces, with $A_0/A_0^+ \otimes_{A_0}$
applied, the identity on $W^A$; cf. Proposition \ref{prop:M2}.    
\end{rem} 

\subsection{}\label{subsec:HCP6} 
We wish to construct a quasi-inverse of the functor given by Proposition \ref{prop:functor-overline}. 
Let $(C, W) \in \mathsf{HCP}$, and
set $J = C^{\circ}$, $V = W^*$. Then $(J, V) \in \mathsf{DHCP}$. Recall $\mathcal{H}(J, V)$,
$\widehat{\mathcal{A}}(C, W)$ and their pairing
$\langle \hspace{2mm}, \ \rangle : \mathcal{H}(J, V) \times \widehat{\mathcal{A}}(C, W) \to \Bbbk$; 
see \eqref{mathcalH},  \eqref{hatmathcalA} 
and \eqref{pairing-complete}. 
Recall also the Hopf super-ideal $I(J, V)$ of $\mathcal{H}(J, V)$ which defines $H(J, V)$ by
$H(J, V)= \mathcal{H}(J, V)/I(J, V)$; see \eqref{H(J,V)}. We define
\begin{equation}\label{A(C,W)}
A(C, W) := \{ Z \in \widehat{\mathcal{A}}(C, W) \mid \langle I(J, V),\ Z \rangle = 0 \}. 
\end{equation}
 
\begin{lemma}\label{lem:A(C,W)1} 
Keep the notation as above. 

(1) In the complete topological Hopf superalgebra $\widehat{\mathcal{A}}(C, W)$, $A(C, W)$ is a super-subalgebra, 
and is stable
under the antipode $\widehat{S}$. 

(2) $A(C, W)$ is discrete in $\widehat{\mathcal{A}}(C, W)$, so that we have
$$A(C, W) \otimes A(C, W) = A(C, W) ~ \widehat{\otimes} ~ A(C, W) \subset 
\widehat{\mathcal{A}}(C, W) ~ \widehat{\otimes} ~ \widehat{\mathcal{A}}(C, W).$$
Moreover, the coproduct $\widehat{\Delta}$ on $\widehat{\mathcal{A}}(C, W)$ induces 
$A(C, W) \to A(C, W) \otimes A(C, W)$.

(3) Given the structure induced from $\widehat{\mathcal{A}}(C, W)$ as above, $A(C, W)$ is a super-commutative
Hopf superalgebra. 
\end{lemma}
\begin{proof}
(1) This follows since $I(J, V)$ is a super-coideal, and is stable under the antipode. 

(2) Since $H(J, V)$ is finitely generated as a left $J$-module, it follows that if we set $M = \bigoplus_{i < n}
J \otimes T^i(V)$ for $n$ large enough, then $\mathcal{H}(J, V) = M + I(J, V)$. If follows that $A(C, W)$
trivially intersects with $\prod_{i \ge n}C\otimes T^n(W)$; this shows that $A(C, W)$ is discrete
since this product is an open neighborhood of 0. 

For simplicity let us write $I$, $\mathcal{H}$, $\widehat{\mathcal{A}}$ for $I(J, V)$, $\mathcal{H}(J, V)$, 
$\widehat{\mathcal{A}}(C,W)$. Consider the pairing on $\mathcal{H} \times \widehat{\mathcal{A}}$ and the induced one
on $\mathcal{H}\otimes \mathcal{H} \times \widehat{\mathcal{A}}~\widehat{\otimes}~\widehat{\mathcal{A}}$. 
Given a subset 
$F$ in $\mathcal{H}$ or in $\mathcal{H}\otimes \mathcal{H}$, let $F^{\bot}$ denote the subset of $\widehat{\mathcal{A}}$
or of $\widehat{\mathcal{A}}~\widehat{\otimes}~\widehat{\mathcal{A}}$ consisting of those elements which annihilate $F$
with respect to the relevant pairing. Then $I^{\bot} = A(C, W)$. 
We see that in $\widehat{\mathcal{A}}~\widehat{\otimes}~\widehat{\mathcal{A}}$,
\begin{align*}
&(\mathcal{H} \otimes I + I \otimes \mathcal{H})^{\bot} 
= (\mathcal{H} \otimes I)^{\bot} \cap (I \otimes \mathcal{H})^{\bot}\\ 
&= \widehat{\mathcal{A}}~\widehat{\otimes}~I^{\bot} \, \cap~I^{\bot} \widehat{\otimes}~\widehat{\mathcal{A}} =
I^{\bot}~\widehat{\otimes}~I^{\bot} = I^{\bot} \otimes I^{\bot}.
\end{align*}
Here the last equality holds true since $I^{\bot} = A(C, W)$ is discrete, as was shown above. 
The desired result follows since $I$ is an ideal, and so
$\widehat{\Delta}( I^{\bot} ) \subset (\mathcal{H} \otimes I + I \otimes \mathcal{H})^{\bot}$.  

(3) This follows easily from the results just proven. 
\end{proof}

Keep the notation as above. For the following two results, set 
$$A = A(C,W),\quad \mathcal{A} = \mathcal{A}(C, W),\quad \widehat{\mathcal{A}} = \widehat{\mathcal{A}}(C, W).$$
Choose a totally ordered basis $X$ of $V$, and
define a unit-preserving super-coalgebra map $\widetilde{\iota}_X : \wedge(V) \to T(V)$, and a counit-preserving
superalgebra map $\widetilde{\pi}_X : T_c(W) \to \wedge(W)$, analogous to $\iota_X$, $\pi_X$ (see  
\eqref{iota}, \eqref{pi}), as follows: 
\begin{align}\label{tildeiotapi}
\widetilde{\iota}_X(x_1 \wedge \dots \wedge x_n \rangle &= 
x_1 \otimes \dots \otimes x_n, \\
\langle \widetilde{\iota}_X(u),\ Z \rangle &=  \langle u,\ \widetilde{\pi}_X(Z) \rangle,
\end{align}
where $x_i \in X,\ x_1 <\dots < x_n,\ 0\le n \le \#X,\ u \in \wedge(V),\ Z \in T_c(W)$. 
Vanishing on $\prod_{n>\#X}T^n(W)$, $\widetilde{\pi}_X$ is continuous, where $\wedge(W)$ is regarded as to
be discrete, and hence it gives rise to the complete tensor product 
$\operatorname{id}_C \widehat{\otimes} ~ \widetilde{\pi}_X : C ~ \widehat{\otimes} ~ T_c(W) \to C 
\otimes \wedge(W)$
of continuous linear maps. 

\begin{lemma}\label{lem:A(C,W)2}
(1) Let 
\begin{equation}\label{psi-prime}
\psi'_X : A \hookrightarrow \widehat{\mathcal{A}} = C ~ \widehat{\otimes} ~ T_c(W) \to C \otimes \wedge(W)
\end{equation}
be the composite of the inclusion with $\operatorname{id}_C \widehat{\otimes} ~ \widetilde{\pi}_X$. 
Then this is a counit-preserving isomorphism of superalgebras, such that
$$ \langle h \otimes \widetilde{\iota}_X(u), \ a \rangle = \langle h \otimes x, \ \psi'_X(a) \rangle, 
\quad h \in J,\ u \in \wedge(V),\ a \in A$$ 

(2) We have $A = A(C, W) \in \mathsf{AHSA}$. 
\end{lemma}
\begin{proof}
(1) Recall from Proposition \ref{prop:Hopf-struc} the isomorphism 
$\xi : \widehat{\mathcal{A}}(C,W) \overset{\simeq}{\longrightarrow} \operatorname{Hom}_J(\mathcal{H}(J, V), C)$.
By using Lemma \ref{lem:Cproper}, one sees that this $\xi$ induces an isomorphism 
$A(C,W) \overset{\simeq}{\longrightarrow} \operatorname{Hom}_J(H(J, V), C)$; this is a key ingredient borrowed from
Koszul \cite{Koszul}. 
Via these isomorphisms together with the 
one $\operatorname{Hom}_J(J \otimes \wedge(V), C)  \overset{\simeq}{\longrightarrow} C \otimes \wedge(W)$
that arises from the canonical pairing (see \eqref{cano-pairing}), the map $\psi'_X$ is identified with 
$\operatorname{Hom}_J(\phi'_X, C)$, where $\phi'_X$ is the composite
$J \otimes \wedge(V) \to J \otimes T(V) = \mathcal{H}(J, V) \to H(J, V)$ of 
$\operatorname{id}_J \otimes \widetilde{\iota}_X$ with the quotient map. Since $\phi'_X$ is an isomorphism by
Lemma \ref{lem:J-free}, the desired result follows. 

(2) This follows since the isomorphism just obtained shows that $A$ is finitely generated.
\end{proof}

From Remark \ref{grA} and \eqref{grhatA}, recall the construction of  
$$\operatorname{gr} A= \bigoplus_{n \ge 0} I_A^n/I_A^{n+1} \simeq \overline{A} \cmdblackltimes \wedge(W^A), 
\quad \operatorname{gr} \widehat{\mathcal{A}} =
\bigoplus_{n = 0}^{\infty} \widehat{\mathcal{I}}_n/\widehat{\mathcal{I}}_{n+1} = \mathcal{A} 
= C \cmdblackltimes T_c(W) , $$ 
where we set $\widehat{\mathcal{I}}_n = \prod_{i \ge n}C \otimes T^i(W)$ in $\widehat{\mathcal{A}}$. 
One sees that $I_A\subset \widehat{\mathcal{I}}_1$, whence $I_A^n \subset \widehat{\mathcal{I}}_1^n 
\subset \widehat{\mathcal{I}}_n$ for every
$n > 0$. Therefore 
the inclusion $A \hookrightarrow \widehat{\mathcal{A}}$
induces, with $\operatorname{gr}$ applied, a Hopf superalgebra map $\operatorname{gr} A \to \mathcal{A}$ which preserves
the $\mathbb{N}$-grading. 

\begin{prop}\label{prop:A(C,W)} 
(1) The $\mathbb{N}$-graded map just obtained is isomorphic in degrees  0, 1, so that we have a Hopf algebra
isomorphism $ \overline{A} \overset{\simeq}{\longrightarrow} C$, and a $C$-colinear isomorphism  
$W^A \overset{\simeq}{\longrightarrow} W$,  
where the right $\overline{A}$-comodule $W^A$ is regarded as 
a right $C$-comodule along the first isomorphism.

(2) The canonical pairing on $\mathcal{H}(J, V) \times \widehat{\mathcal{A}}(C, W)$ given in \eqref{pairing-complete}
induces a non-degenerate Hopf pairing
\begin{equation}\label{HApair} 
\langle \hspace{2mm}, \ \rangle : H(J, V) \times A(C, W) \to \Bbbk,
\end{equation}
which induces an isomorphism $H(J, V) \overset{\simeq}{\longrightarrow} A(C, W)^{\circ}$ of Hopf superalgebras. 
The Hopf pairing and the induced isomorphism are both natural in $(C, W)$. 

(3) The pair of the isomorphisms obtained in Part 1 gives a natural isomorphism $(\overline{A}, W^A) 
\overset{\simeq}{\longrightarrow} (C, W)$ in $\mathsf{HCP}$.   
\end{prop}
\begin{proof}
(1) Apply to the trivial $\operatorname{gr}$ construction to $C \otimes \wedge(W)$ with respect to the
descending chain $\bigoplus_{i\ge n}C \otimes \wedge^i(W)$, $n = 0, 1,\dots$, of super-ideals; these 
super-ideals are the 
powers of the kernel of the projection $C \otimes \wedge(W) \to C$. The resulting 
$\operatorname{gr}(C \otimes \wedge(W))$ coincides with the original $C \otimes \wedge(W)$. 
Commuting with the projection onto $C$, $\operatorname{id}_C \widehat{\otimes} ~ \widetilde{\pi}_X$ induces an 
$\mathbb{N}$-graded algebra map, $\operatorname{gr} \widehat{\mathcal{A}} = \mathcal{A} \to 
\operatorname{gr}(C \otimes \wedge(W)) = C \otimes \wedge(W)$, which is seen to be isomorphic in degrees 0, 1.
Since $\operatorname{gr} \psi'_X$ is an isomorphism by Lemma \ref{lem:A(C,W)2}(1), the desired result follows.  

(2) Since the tensor product of the canonical pairings on $J \times C$, $\wedge(V) \times \wedge(W)$ is 
non-degenerate, and induces an isomorphism 
$J \otimes \wedge(V) \overset{\simeq}{\longrightarrow} (C \otimes \wedge(W))^{\circ}$,
the desired result follows from Lemma \ref{lem:A(C,W)2}(1).

(3) Since the two isomorphisms both commute with the natural projections from $A(C, W)$, 
their duals $J = C^{\circ} \overset{\simeq}{\longrightarrow} \overline{A}^{\circ}$, 
$V = W^* \overset{\simeq}{\longrightarrow} (W^A)^*$ 
both extend to the isomorphism $H(J, V) \overset{\simeq}{\longrightarrow} A(C, W)^{\circ}$
obtained in Part 2. 
This shows that they give an isomorphism in $\mathsf{HCP}$.
The naturality is easy to see.  
\end{proof} 

\begin{theorem}\label{thm:equiv}
The assignment $(C, W) \mapsto A(C, W)$ is functorial, so that we have the functor 
$$ A : \mathsf{HCP} \to \mathsf{AHSA},\ (C, W) \mapsto A(C, W).$$ 
This is an equivalence which is a quasi-inverse of the functor $A \mapsto (\overline{A}, W^A)$ given 
by Proposition \ref{prop:functor-overline}. 
\end{theorem}
\begin{proof}
The functoriality is easy to see. Proposition \ref{prop:A(C,W)}(3) shows that the functor $A$ followed by
$A \mapsto (\overline{A}, W^A)$ is isomorphic to the identity functor on $\mathsf{HCP}$.

To complete the proof let $A\in \mathsf{AHSA}$, and set $(C, W) := (\overline{A}, W^A)$ in $\mathsf{HCP}$.
Let $ \varpi : A \to A_1 \to A_1/A_0^+A_1 = W$ be the composite of the canonical projections. 
For each $n > 0$, define $\varpi^{(n)} : A \to T^n(W)$ by
$$ \varpi^{(n)}(a) = \sum \varpi(a_{(1)}) \otimes \varpi(a_{(2)}) \otimes \dots \otimes \varpi(a_{(n)}). $$
For $n=0$, let $\varpi^{(0)} = \varepsilon$. 
Define a map by
$$ \beta_A : A \to \widehat{\mathcal{A}}(C, W) = C ~ \widehat{\otimes} ~ T_c(W), \ \beta_A(a) = 
\sum_{n=0}^{\infty}\sum \overline{a}_{(1)} \otimes \varpi^{(n)}(a_{(2)}). $$
This is natural in $A$. It suffices to prove that $\beta_A$ gives a Hopf superalgebra isomorphism from $A$ 
onto $A(C, W)$. 

Suppose that $(J, V) \in \mathsf{DHCP}$ is the pair associated with $(C, W)$. Let $H = A^{\circ}$. 
By (the proof of) Proposition \ref{prop:functor-overline} we have $(J, V) = (\underline{H}, V_H)$. 
Let $\gamma : \mathcal{H}(J, V) \to H(J, V) \overset{\simeq}{\longrightarrow} H$ be the composite 
of the quotient map with the natural isomorphism $\alpha_H$ given in \eqref{alpha}; this is a Hopf
superalgebra map. One sees that
$$\langle \gamma(h \otimes u),\ a \rangle = \langle h \otimes u, \ \beta_A(a) \rangle,\quad h \in J, \ 
u \in T(V),\ a \in A.$$
It follows from Corollary \ref{cor:non-degenerate}
that $\beta_A$ is an injective Hopf superalgebra map into $A(C, W)$. 
The image of $\beta_A$ is precisely $A(C, W)$, as desired, since given a totally ordered basis $X$ of $V$, the
composite $A \to A(C, W) \overset{\simeq}{\longrightarrow} C \otimes \wedge(W)$ of 
$\beta_A$ with $\psi'_X$ coincides with the isomorphism $\psi_X$.  
\end{proof} 

\begin{definition}\label{def:connected-super}
An affine Hopf superalgebra $A$ is said to be \emph{connected} if $\overline{A}$, or equivalently
the prime spectrum $\operatorname{Spec} A_0$ of $A_0$, is connected. An algebraic affine supergroup
scheme $G = \operatorname{SSp} A$ is said to be \emph{connected}, if $A$ is connected. 
\end{definition}
 
Obviously, the equivalence obtained above restricts to a category equivalence between $\mathsf{cHCP}$
and the full subcategory of $\mathsf{AHSA}$ consisting of all connected ones.

\section{Duality and short exact sequences}\label{sec:duality}

\subsection{}\label{LieHyper}
Let $G = \operatorname{SSp} A$ be an algebraic affine supergroup scheme. 
Set $H = A^{\circ}$. Then $H \in \mathsf{CCHSA}$. Moreover, the primitives $P(H)$ form a finite-dimensional 
Lie superalgebra, and the irreducible component $H^1$ containing 1 is a hyper-superalgebra; 
see Section \ref{subsec:corad}. 
We write
$$ \operatorname{Lie} G = P(H),\quad \operatorname{hy} G = H^1,$$
and call these the \emph{Lie superalgebra}, and the \emph{hyperalgebra} of $G$, respectively.
One sees that $\operatorname{Lie},\ \operatorname{hy}$ give functors, which will play an important role in the following 
section.

\begin{prop}\label{prop:LieHyper}
Suppose that an algebraic affine supergroup scheme $G$ is represented by $A(C, W)$, where 
$(C, W) \in \mathsf{HCP}$. Let $(J, V) = ((C^{\circ})^1, W^*)$ be the object in $\mathsf{iDHCP}$ 
which is associated, by the functor given in Proposition \ref{prop:dualization}(2), with $(C, W)$. 

(1) $\operatorname{Lie} G$ is naturally isomorphic to the Lie superalgebra constructed on $P(J) \oplus V$
as in Remark \ref{rem:DHCPzero-char}(1).

(2) $\operatorname{hy} G$ is naturally isomorphic to $H(J, V)$. 
\end{prop}

\begin{proof}
This follows immediately from the isomorphism given in Proposition \ref{prop:A(C,W)}(2). For Part 2 note
that $H(C^{\circ}, V)^1 = H(J, V)$. 
\end{proof}

\subsection{}\label{subsec:dual-of-H(J,V)}
Conversely, let us start with an object $(J, V) \in \mathsf{DHCP}$, where we assume 
$\dim V < \infty$. Construct $H(J, V) \in \mathsf{CCHSA}$. 

\begin{prop}\label{prop:dual-of-H(J,V)}
With the notation as above, set $A = H(J, V)^{\circ}$; this is a super-commutative Hopf superalgebra
which is not necessarily finitely generated. 

(1) The restriction maps $A \to J^{\circ}$, $A \to V^*$ induce isomorphisms
$$ \overline{A} \overset{\simeq}{\longrightarrow} J^{\circ},\quad W^A \overset{\simeq}{\longrightarrow} V^*$$
of Hopf algebras, and of super-vector spaces, respectively. 

(2) Choose arbitrarily a totally ordered basis $X$ of $V$, and let 
$\phi_X : J \otimes \wedge(V) \overset{\simeq}{\longrightarrow} H(J, V)$ be the isomorphism as given by
Proposition \ref{prop:M1}.
Then we have uniquely a counit-preserving, left $\overline{A}$-colinear superalgebra isomorphism 
$\psi_X : A \overset{\simeq}{\longrightarrow} \overline{A} \otimes \wedge(W^A)$ that satisfies the same formula as
\eqref{phi-psi}, when we replace the pairings on $\underline{H} \times \overline{A}$, $V_H \times W^A$ in 
\eqref{phi-psi}
with those on $J \times \overline{A}$, $V \times W^A$ which are obtained from the isomorphisms of Part 1 above. 
\end{prop}
\begin{proof}
Let $\phi_X$ be the isomorphism as in Part 2, and identify $H(J, V)$ with $J \otimes \wedge(V)$ via $\phi_X$.
By \cite[Proposition 3.9(2)]{M}, $H(J, V)$ has the natural filtration $F_n := \bigoplus_{i \le n}J \otimes \wedge^i(V)$, $n = 0, 1, \dots$,
which is compatible with the structure so that in particular, $F_nF_m \subset F_{n+m}$ for all $n, m \ge 0$.
It follows that if $I$ is a right ideal of $J$ (of cofinite dimension), then $I \otimes \wedge(V)$ is a right
ideal of $H(J,V)$ (of cofinite dimension). Note that in general, the dual coalgebra $R^{\circ}$ of an algebra $R$
consists of those elements in $R^*$ which annihilate some cofinite-dimensional right ideal of $R$, since
such a right ideal, say $\mathfrak{a}$, includes a cofinite-dimensional ideal; take the annihilators 
$\operatorname{Ann}(R/\mathfrak{a})$, for example. Then we see that $H(J, V)^{\circ} = J^{\circ} \otimes (\wedge(V))^*
\simeq J^{\circ} \otimes \wedge(V^*)$, which proves the proposition. 
\end{proof}

\begin{rem}\label{rem:dual-of-H(J,V)}
Let $(J, V) \in \mathsf{DHCP}$ with $\dim V < \infty$, as above.  
Set $C := J^{\circ}$, $W := V^*$. Then the right $J$-module structure on $V$ 
gives rise 
naturally to a right $C$-comodule structure on $W$ (see \eqref{J-mod-struc}), from which we can construct, 
as before, a complete topological Hopf superalgebra, $\widehat{\mathcal{A}}(C, W)$, together with a pairing, 
$\langle \hspace{2mm}, \ \rangle : \mathcal{H}(J, V) \times \widehat{\mathcal{A}}(C, W) \to \Bbbk$. If we define
$A(C,W)$ by the same formula as \eqref{A(C,W)}, then it follows, as before, that $A(C,W)$ is a super-commutative
Hopf superalgebra. Note that during the argument in the preceding section, 
what was essentially used from Lemma \ref{lem:Cproper} was not
the non-degeneracy of the pairing $\langle \hspace{2mm}, \ \rangle : J \times C \to \Bbbk$, but rather only
the injectivity of $C \to J^*$, and that this last injectivity holds true in the present situation. Then we see from 
the proof of Theorem \ref{thm:equiv} that $H(J,V)^{\circ}$ is naturally isomorphic to the $A(C,W)$ just defined.  
\end{rem}

\subsection{}\label{subsec:exact-sequence}
Given a map $q : H \to H'$ of Hopf superalgebras, we define 
\begin{align*}
\quad H^{\mathrm{co}q} &:= \{x\in H\mid \sum x_{(1)} \otimes q(x_{(2)}) = x \otimes q(1)\}, \\
\text{resp.,} \ {}^{\mathrm{co}q} H &:= \{x\in H\mid \sum q(x_{(1)}) \otimes x_{(2)} = q(1) \otimes x \}.
\end{align*}
This is the left (resp., right) coideal super-subalgebra of $H$ consisting of all \emph{right} 
(resp., \emph{left}) \emph{coinvariants}
along $q$. If $H$ is super-cocommutative, then ${}^{\mathrm{co}q} H = H^{\mathrm{co}q}$, and this is
a Hopf super-subalgebra of $H$. 

Let $H \in \mathsf{CCHSA}$. Recall that a Hopf super-subalgebra $K$ of $H$
is said to be \emph{normal} \cite[p.295]{M}, if $HK^+ = K^+H$, or equivalently 
if $K$ is stable under the adjoint action, that is,
$$ \sum (-1)^{|x_{(1)}||y|}S(x_{(1)}) y \, x_{(2)} \in K, \quad x \in H, y \in K; $$
see \cite[Theorem 3.10]{M} for the equivalence of the conditions. 
If this is the case, $H/HK^+$ is a quotient Hopf superalgebra of $H$.
We write $H//K$ for $H/HK^+$. 

We say that an injective morphism $K \overset{p}{\to} H$ in $\mathsf{CCHSA}$
is \emph{normal} if the image $p(K)$ of $K$ is normal in $H$. 

\begin{definition}\label{def:exact-seq-cocom}
A sequence $H_1 \overset{p}{\to} H_2 \overset{q}{\to} H_3$ in $\mathsf{CCHSA}$ is said to be
\emph{short exact} if the following equivalent conditions are satisfied:
\begin{itemize}
\item[(a)] $p$ is an injection, and is normal, and $q$ induces an isomorphism
$H_2//p(H_1) \overset{\simeq}{\longrightarrow} H_3$; 
\item[(b)] $q$ is a surjection, and $p$ is an isomorphism $H_1 \overset{\simeq}{\longrightarrow} H_2^{\mathrm{co}q}$
onto $H_2^{\mathrm{co}q}$.  
\end{itemize}
The equivalence above follows from \cite[Theorem 3.10(3)]{M}. 
\end{definition}

\begin{prop}\label{prop:exact-seq-cocom}
(1) The morphism $H(f,g) : H(J_1, V_1) \to H(J_2, V_2)$ in $\mathsf{CCHSA}$ which arises from a morphism
$(f,g) : (J_1,V_1) \to (J_2,V_2)$ in $\mathsf{DHCP}$ is an injection (resp., a surjection) if and only if
$f : J_1 \to J_2$ and $g : V_1 \to V_2$ are both injections (resp., surjections). 

(2) The injective morphism $H(f,g) : H(J_1, V_1) \to H(J_2, V_2)$ in $\mathsf{CCHSA}$ which arises from a 
morphism $(f,g) : (J_1,V_1) \to (J_2,V_2)$ in $\mathsf{DHCP}$, with $f$, $g$ injections, is normal if and only if
(i) $f : J_1 \to J_2$ is normal, (ii) $g(V_1)$ is $J_2$-stable in $V_2$, (iii) $[g(V_1), V_2] \subset f(J_1)$ and 
(iv) $v \triangleleft f(a) - \varepsilon(a)v \in g(V_1)$ for all $a \in J_1$, $ v \in V_2$.  

(3) The sequence $H(J_1,V_1)\to H(J_2, V_2) \to H(J_3, V_3)$ in $\mathsf{CCHSA}$ which arises from a 
sequence $(J_1, V_1) \to (J_2, V_2) \to (J_3, V_3)$ in $\mathsf{DHCP}$ is short exact if and only if
the sequences $J_1 \to J_2 \to J_3$, $V_1 \to V_2 \to V_3$ which constitute the latter are short
exact sequences of cocommutative Hopf algebras, and of vector spaces, respectively.  
\end{prop}
\begin{proof}
(1) This follows from Proposition \ref{prop:M1}; see also \cite[Remark 3.8]{M}. 

(2) This is seen from the construction of $H(J, V)$. Note that Condition (iv) is equivalent to that 
$f(a)v - vf(a) \in \operatorname{Im} H(f,g)$ for all $a \in J_1$, $ v \in V_2$. 

(3) This follows by Part 1 and \cite[Theorem 3.13(3)]{M}.
 \end{proof}

\subsection{}\label{subsec:exact-seq-comm}
Let $A$ be a super-commutative Hopf superalgebra. For every Hopf super-subalgebra $B \subset A$, 
$AB^+ = B^+A$, and this is a Hopf super-ideal of $A$, so that we have a quotient Hopf superalgebra 
$A/AB^+$ of $A$;  we write $A // B$ for $A/AB^+$.
 
Let $q : A \to B$ be a surjective morphism of super-commutative Hopf superalgebras. This is said to be 
\emph{conormal}
\cite[Definition 5.7]{M}, if ${}^{\mathrm{co}q}A = A^{\mathrm{co}q}$, or equivalently if the kernel 
$\operatorname{Ker} q$ is 
$A$-costable, that is, sent to $\operatorname{Ker} q \otimes A$, under the adjoint coaction $A \to A \otimes A$ given by 
$$a \mapsto  \sum (-1)^{|a_{(1)}||a_{(2)}|} a_{(2)} \otimes S(a_{(1)})a_{(3)}; $$
see \cite[Theorem 5.9]{M} for the equivalence of the conditions. 
We have the closed embedding  $N := \operatorname{SSp} B \hookrightarrow G:=\operatorname{SSp} A$ of supergroup schemes
which corresponds to $q$. One sees that $q$ is conormal if and only if for every super-commutative superalgebra $R$,
the subgroup $N(R)$ of $G(R)$ is normal. 

\begin{definition}\label{def:exact-seq-com}
(1) A sequence $A_1 \overset{p}{\to} A_2 \overset{q}{\to} A_3$ of super-commutative Hopf superalgebras is said to be
\emph{short exact} if the following equivalent conditions are satisfied:
\begin{itemize}
\item[(a)] $p$ is an injection, and $q$ induces an isomorphism
$A_2//p(A_1) \overset{\simeq}{\longrightarrow} A_3$; 
\item[(b)] $q$ is a surjection, and is conormal, and $p$ is an isomorphism $A_1 \overset{\simeq}{\longrightarrow}
A_2^{\mathrm{co}q}$ onto $A_2^{\mathrm{co}q}$.  
\end{itemize} 

(2) A sequence $G_1 \overset{\xi}{\to} G_2 \overset{\eta}{\to} G_3$ of affine supergroup schemes is 
said to be \emph{short exact}, if the following equivalent conditions are satisfied: 
\begin{itemize}
\item[(a)] $\eta$ is an epimorhism of dur sheaves, and $\xi$ gives an isomorphism
$G_1 \overset{\simeq}{\longrightarrow} \operatorname{Ker} \eta$; 
\item[(b)] $\xi$ is a closed embedding onto a closed normal super-subgroup, say $G_1'$, of $G_2$, and 
$\eta$ induces an isomorphism $G_2 \tilde{\tilde{/}} G_1' \overset {\simeq}{\longrightarrow} G_3$,
where $G_2 \tilde{\tilde{/}} G_1'$ denotes the dur sheafification of the $`$naive' functor which associates
to every super-commutative algebra $R$, the quotient group $G_2(R)/G_1'(R)$; 
see \cite[Chap.~III, Sect.~3, 7.2]{DG}, \cite[Sect.~4]{Z1}.  
\end{itemize} 
\end{definition}

This pair of equivalences, and the following proposition as well, follow from 
\cite[Theorem 5.9]{M} and \cite[Proposition 4.2]{Z1}. 

\begin{prop}\label{prop:exact-seq-supergroup}
Every short exact sequence $G_1 \to G_2 \to G_3$ of affine supergroup schemes arises uniquely
from a short exact sequence $\mathcal{O}(G_3) \to \mathcal{O}(G_2) \to \mathcal{O}(G_1)$ of super-commutative
Hopf superalgebras, where $\mathcal{O}(G)$ denotes the super-commutative Hopf superalgebra which represents an
affine supergroup scheme $G$.  
\end{prop}

\begin{prop}\label{prop:exact-seq-com}
(1) The morphism $A(f,g) : A(C_1, W_1) \to A(C_2, W_2)$ in $\mathsf{AHSA}$ which arises from a morphism
$(f,g) : (C_1,W_1) \to (C_2,W_2)$ in $\mathsf{HCP}$ is an injection (resp., a surjection) if and only if
$f : C_1 \to C_2$ and $g : W_1 \to W_2$ are both injections (resp., surjections). 

(2) Let $A(f,g) : A(C_2, W_2) \to A(C_3, W_3)$ be a surjective morphism in $\mathsf{AHSA}$ which arises from a 
morphism $(f,g) : (C_2,W_2) \to (C_3,W_3)$ in $\mathsf{HCP}$, with $f$, $g$ surjections. We know from 
Proposition \ref{prop:A(C,W)}(2) that the dual injective morphism $A(f,g)^{\circ}$ in $\mathsf{CCHSA}$ is naturally
identified with $H(f^{\circ}, g^*) : H(C_3^{\circ}, W_3^*) \to H(C_2^{\circ}, W_2^*)$ that arises from
the morphism $(f^{\circ}, g^*)$ in $\mathsf{DHCP}$ associated with $(f,g)$.
Then, $A(f,g)$ is conormal if and only if $f : C_2 \to C_3$ is conormal, and $H(f^{\circ}, g^*)$ is normal. 

(3) The sequence $A(C_1,W_1)\to A(C_2, W_2) \to A(C_3, W_3)$ in $\mathsf{AHSA}$ which arises from a 
sequence $(C_1, W_1) \to (C_2, W_2) \to (C_3, W_3)$ in $\mathsf{HCP}$ is short exact if and only if
the sequences $C_1 \to C_2 \to C_3$, $W_1 \to W_2 \to W_3$ which constitute the latter are short
exact sequences of commutative Hopf algebras, and of vector spaces, respectively.  
\end{prop}
\begin{proof}
(1) This follows from Proposition \ref{prop:M2}; see also \cite[Remark 4.8]{M}. 

Next, we prove first Part 3, and then Part 2.

(3) This follows by Part 1 and \cite[Theorem 5.13(3)]{M}. 

(2) If a closed super-subgroup $N \subset G$ of a affine super-group scheme $G$ is normal, 
then $N_{\operatorname{res}}$
is normal in $G_{\operatorname{res}}$; see \eqref{Gres}. A conormal surjective morphism $q : A \to B$ of
super-commutative Hopf superalgebras induces a normal injective morphism $q^{\circ} : B^{\circ} \to A^{\circ}$
of super-cocommutative Hopf superalgebras, since the condition that the adjoint $A$-coaction is induced onto 
$B$ along $q$ is dualized to that adjoint $A^{\circ}$-action on $A^{\circ}$ stabilizes $B^{\circ}$ 
through $q^{\circ}$. These together with Proposition \ref{prop:A(C,W)}(2) prove the $`$only if' part. 

To prove the $`$if' part, assume the second condition, and set 
$C_1 = C_2^{\mathrm{co}f}$, $W_1 = \operatorname{Ker} g$. 
Note that $C_1$ is an affine Hopf algebra. 
Conditions (ii), (iv) of Proposition \ref{prop:exact-seq-cocom}(2) imply that $W_1$ is a $C_2$-subcomodule of $W_2$, 
and is a right $C_1$-comodule.  We regard $P(C_2^{\circ}) \supset P(C_3^{\circ})$ via $f^{\circ}$. 
Condition (iii) implies that the structure $[\hspace{2mm} , \ ] : W_2^* \times W_2^*
\to P(C_2^{\circ})$ of $(C_2, W_2)$ induces a bilinear map $W_1^* \times W_1^* \to P(C_2^{\circ})/P(C_3^{\circ})$. Since $0 \to P(C_3^{\circ})
\to P(C_2^{\circ}) \to P(C_1^{\circ})$ is exact, we can compose the induced bilinear map with a natural injection 
$P(C_2^{\circ})/P(C_3^{\circ}) \hookrightarrow P(C_1^{\circ})$ to obtain 
$W_1^* \times W_1^* \to P(C_1^{\circ})$, which is $C_1$-colinear, as is easily seen. We see that $(C_1, W_1)$
together with the thus obtained bilinear map forms a Harish-Chandra pair, and is a sub-object of $(C_2, W_2)$.  
By Part 3, $A(f,g)$ extends to a short exact sequence $A(C_1,W_1) \to A(C_2,W_2) \to A(C_3, W_3)$, and hence is
conormal. 
\end{proof}

\section{Simply connected affine supergroup schemes}\label{sec:simply-conn}

\begin{definition}\label{def:simply-conn}
Given a connected algebraic affine supergroup scheme $G$, 
an \emph{etale supergroup covering} of $G$ is a pair $(F, \eta)$ of a connected algebraic affine supergroup 
scheme $F$ and an epimorphism $\eta : F \to G$ of supergroup sheaves 
such that the kernel $\operatorname{Ker} \eta$ of $\eta$ is purely even, 
and is finite etale. A \emph{simply connected affine supergroup scheme} is a connected algebraic 
affine supergroup scheme that has no non-trivial etale supergroup covering. 
\end{definition}

This directly generalizes the definition in the non-super situation; see \cite{T-I}, for example.

\begin{prop}\label{prop:simply-conn}
A connected algebraic affine supergroup scheme $G$ is simply connected if and only if the associated
affine group scheme $G_{\operatorname{res}}$ (see \eqref{Gres}),
which is necessarily algebraic and connected, is simply connected. 
\end{prop}
\begin{proof}
Proposition \ref{prop:exact-seq-com}(3), combined with Proposition \ref{prop:exact-seq-supergroup}, shows that
if $E \to F \to G$ is a short exact sequence of affine supergroup schemes with $E$ purely even, then
it induces the short exact sequence $E \to F_{\operatorname{res}} \to G_{\operatorname{res}}$ 
of affine group schemes. 
This proves the $`$if' part. 

To prove the $`$only if' part, suppose that we are given a short exact sequence of affine group schemes
$E \to F_0 \to G_{\operatorname{res}}$ in which $E$ is finite etale, and $F_0$ is connected algebraic.  
It suffices to construct such a short exact sequence $E \to F \to G$ of affine supergroup schemes
that induces the given short exact sequence. 
By \cite[Proposition 1.1]{T-I}, $F_0 \to G$ induces an isomorphism $\operatorname{hy} F_0
\overset{\simeq}{\longrightarrow} \operatorname{hy} G_{\operatorname{res}}$. 
Suppose that $G$ corresponds to $(C, W) \in \mathsf{cHCP}$. Then we may suppose that the Hopf algebra
$D :=\mathcal{O}(F_0)$ which represents $F_0$ includes $C$ as a Hopf subalgebra, 
so that $W$ may be regarded a right $D$-comodule.
We see from the last isomorphism and Remark \ref{rem:HCP}(1)       
that the structure $[\hspace{2mm} , \ ]$ of $(C, W)$ makes $(D, W)$ into an object
in $\mathsf{cHCP}$ including $(C, W)$ as a sub-object. If $F$ denotes the algebraic affine
supergroup scheme corresponding to $(D, W)$, there arises a desired short exact sequence, 
again by Proposition \ref{prop:exact-seq-com}(3).  
\end{proof}

\begin{theorem}\label{thm:simply-conn-zero}
Suppose that $\Bbbk$ is an algebraically closed field of characteristic zero. Then $G \mapsto \operatorname{Lie} G$
gives rises to a bijection from the set of the isomorphism classes of 
all simply connected affine supergroup schemes $G$ to the set of the isomorphism classes 
of all those finite-dimensional Lie superalgebras $L$ such that (i)~the radical $\operatorname{Rad} L_0$ of $L_0$ is
nilpotent, and (ii)~$\operatorname{ad}^n(\operatorname{Rad} L_0)(L_1) = 0$ for some $n>0$.  
\end{theorem}
\begin{proof}
In the non-super situation this was proved by Hochschild \cite{H}. According to \cite{H}, 
a connected algebraic affine group scheme $G = \operatorname{Sp} C$ with $L_0 = \operatorname{Lie} G$ 
is simply connected if and only if (a)~the radical of $L_0 =P(C^{\circ})$ is nilpotent, and 
(b)~the canonical injective Hopf algebra map  $C \to U(L_0)^{\circ}$
maps onto the Hopf subalgebra $B(L_0)$ of $U(L_0)^{\circ}$ consisting of those elements
which annihilate some powers of the ideal $\langle \operatorname{Rad} L_0 \rangle$ of $U(L_0)$ generated by 
the radical $\operatorname{Rad} L_0$. Moreover, $C \mapsto P(C^{\circ})$ and $L_0 \mapsto B(L_0)$ give
a bijection, modulo isomorphism, between the set of all those connected affine Hopf algebras $C$ which 
satisfy the conditions (a), (b) and the set of all finite-dimensional Lie algebras $L_0$ with nilpotent radical. 
Suppose that $C$ and $L_0$ correspond to each other. Then, given a finite-dimensional vector space $W$,
all the right $C$-comodule structures on $W$ and all those right $U(L_0)$-module structures on $W^*$ such that
$W^* \triangleleft \langle\operatorname{Rad} L_0 \rangle^n= 0$ for some $n>0$ are 
naturally in one-to-one correspondence. It follows
by Remark \ref{rem:HCP}(1)
that $(C, W) \mapsto P(C^{\circ}) \oplus W^*$ gives a bijection, modulo isomorphism, from the set of
all those object $(C, W)\in \mathsf{cHCP}$ such that $\operatorname{Sp} C$ is simply connected to the set of
all those finite-dimensional Lie superalgebras $L$ which satisfy the conditions (i), (ii) above. 
This, combined with Theorem \ref{thm:equiv} and Proposition \ref{prop:simply-conn}, proves
the desired result. 
\end{proof}

\begin{lemma}\label{lem:proper}
A super-cocommutative Hopf superalgebra $H$ such that $\dim V_H < \infty$ is proper if and only 
if $\underline{H}$ is proper.
Here, recall that an algebra $R$ is said to be proper if the canonical map 
$R \to (R^{\circ})^*$ is an injection. 
\end{lemma}
\begin{proof}
This is seen from Proposition \ref{prop:dual-of-H(J,V)}. 
\end{proof}

\begin{definition}\label{def:finite-type}
Generalizing directly the definition \cite[p.258]{T-I} in the non-super situation, we say that a hyper-superalgebra $H$
is \emph{of finite type}, if the Lie superalgebra $P(H)$ is finite-dimensional. 
This is equivalent to the condition that the hyperalgebra
$\underline{H}$ is of finite-type, and $\dim V_H < \infty$, since we have $P(H) = P(\underline{H}) \oplus V_H$.   
\end{definition}

\begin{theorem}\label{thm:simply-conn-p}
Suppose that $\Bbbk$ is a perfect field of positive characteristic $>2$. Then $G \mapsto \operatorname{hy} G$
gives an equivalence from the category of the simply connected affine supergroup schemes to the category
of all those hyper-superalgebras $H$ of finite type such that (i)~$\underline{H}^{\operatorname{ab}}$ 
is finite-dimensional, and (ii)~$H$ is proper.  Here, $\underline{H}^{\operatorname{ab}}$ is the abelianization of $\underline{H}$,
that is, the quotient (indeed, Hopf) algebra of $\underline{H}$ divided by its commutator ideal. 
\end{theorem}

\begin{proof}
In the non-super situation this was proved by Takeuchi \cite[Corollary 5.4]{T-III}. 
According to \cite{T-III}, 
a connected algebraic affine group scheme $G = \operatorname{Sp} C$ with $J = \operatorname{hy} G$ 
is simply connected if and only if the canonical Hopf algebra map  $C \to J^{\circ}$ is an isomorphism. 
Moreover, $C \mapsto (C^{\circ})^1$ and $J \mapsto J^{\circ}$ give an anti-equivalence 
between the category of all those connected affine Hopf algebras $C$ such that $\operatorname{Sp} C$ 
is simply connected and the category of all those hyperalgebras $J$ of finite type such that 
$J^{\operatorname{ab}}$ is finite-dimensional, and $J$ is proper.  
We see from Remark \ref{rem:HCP}(1)
that $(C, W) \mapsto ((C^{\circ})^1, W^*)$ and $(J, V) \mapsto (J^{\circ}, V^*)$
give an anti-equivalence between the category of all those objects in $\mathsf{cHCP}$ such that
$\operatorname{Sp} C$ is simply connected and the category of all those objects $(J, V)$ in $\mathsf{iDHCP}$
such that $J$ is of finite-type and proper, and $J^{\operatorname{ab}}$ is finite-dimensional.
This, combined with Theorems \ref{thm:equiv}, \ref{thm:Takeuchi}(2),  Proposition \ref{prop:simply-conn} 
and Lemma \ref{lem:proper}, proves the desired result. In fact, we see from Proposition 
\ref{prop:dual-of-H(J,V)} that
$H \mapsto \operatorname{SSp} H^{\circ}$ gives a quasi-inverse of $G \mapsto \operatorname{hy} G$. 
\end{proof}

\section{Unipotent affine supergroup schemes}\label{sec:unipotent}

\begin{prop}\label{prop:odd-primitive}
Let $(C, W) \in \mathsf{HCP}$, and set $A = A(C, W)$ in $\mathsf{AHSA}$. If there exists a non-zero
element $w \in W$ on which $C$ trivially coacts, that is,  $w \mapsto w \otimes 1$ via the
$C$-comodule structure $W \to W \otimes C$, then $A$ contains a non-zero odd primitive. 
In particular, if $C$ is irreducible and if $W \ne 0$, the conclusion holds true. 
\end{prop}
\begin{proof}
Let $(J, V) = (C^{\circ}, W^*)$ be the object in $\mathsf{DHCP}$ associated with $(C, W)$. 
For an element $w \in W$ as above, let $U$ be the subspace
of $V$ consisting of those elements which annihilate $w$. The one sees that $U$ is a $J$-submodule
of $V$, so that we have a sub-object $(J, U)$ of $(J, V)$. One sees that if $a \in J, v \in V$, then
$\langle v \triangleleft a, w \rangle = \varepsilon(a)\, \langle v, w \rangle$, and so 
$v \triangleleft a -\varepsilon(a)v \in U$. 
By Proposition \ref{prop:exact-seq-cocom},
$H(J, U)$ is normal in
$H(J, V)$, and the corresponding quotient Hopf superalgebra $H(J, V)//H(J, U)$ is
isomorphic to $\wedge(\Bbbk x)$, which is generated by a non-zero odd primitive $x$ with $x^2 = 0$. 
Those elements in $A$ which annihilate, with respect to the pairing \eqref{HApair}, the Hopf super-ideal
$I := H(J, V)H(J, U)^+$ form a Hopf super-subalgebra, say $B$, which is naturally embedded into 
$(\wedge(\Bbbk x))^*$.  We claim that this last embedding is an isomorphism; this obviously implies the 
desired result. Choose a totally ordered basis $Y$ of $U$, and add to it a final element, say $z$, to obtain
a totally ordered basis, $X = Y \cup \{ z \}$, of $V$. Let $\phi_X : J \otimes \wedge(V)
\overset{\simeq}{\longrightarrow} H(J, V)$ be the isomorphism as given in Proposition \ref{prop:M1}. 
With \cite[Remark 3.8]{M} applied to $H(J, V) \to H(J, V)/I$, we see that $\phi_X^{-1}(I) =
(J \otimes \wedge(U))^+ \otimes \wedge( \Bbbk z )$; this, together 
with Proposition \ref{prop:A(C,W)}(2) (see its proof), proves the claim.
\end{proof}

Recall from Definition \ref{def:unipotent-reductive}(1) the definition of 
unipotent affine supergroup schemes.

\begin{lemma}\label{lem:unipotency}
Given a short exact sequence $N \to G \to F$ of affine supergroup schemes, $G$ is unipotent if and only if
$N$ and $F$ are both unipotent.
\end{lemma}
\begin{proof}
The $`$only if' part follows since if the Hopf superalgebra $\mathcal{O}(G)$ which represents $G$ is irreducible, then
$\mathcal{O}(F)$ and $\mathcal{O}(N)$, being its sub- and quotient coalgebras, are irreducible. 

To prove the $`$if' part, let $V$ be a simple rational $G$-supermodule. The $N$-invariants $V^N$ in $V$ form
a $G$-super-submodule, since $N \triangleleft G$.  
If $N$ is unipotent, then $V^N$ is non-zero, and coincides with $V$ by the simplicity. Therefore, $V$ may
be regarded as a simple rational $F$-supermodule, and is trivial as an $F$- and $G$-supermodule if $F$ is unipotent. 
\end{proof}

The following theorem is due to Alexandr Zubkov, who informed privately the author of his proof.
Here we give a simpler proof. 

\begin{theorem}[A.~Zubkov]\label{thm:unipotency}
An affine supergroup scheme $G$ is unipotent if and only if the associated affine group scheme 
$G_{\operatorname{res}}$ is unipotent.
\end{theorem}
\begin{proof}
Suppose $G = \operatorname{SSp} A$. Recall $G_{\operatorname{res}} = \operatorname{Sp} \overline{A}$.
Then the $`$only if' part follows since if $A$ is irreducible, then
so is the quotient coalgebra $\overline{A}$.

To prove the $`$if' part,  recall that $A$ is a union $\bigcup_{\alpha}A_{\alpha}$ of affine Hopf super-subalgebras
$A_{\alpha}$. Since $\overline{A} = \bigcup_{\alpha}\overline{A}_{\alpha}$ by 
\cite[Proposition 4.6(3)]{M}, we have only to prove, and will prove by induction on $\dim W^A$, that
an affine Hopf superalgebra $A$ is irreducible, assuming that $\overline{A}$ is irreducible.  
We may suppose that $A$ is not purely even. 
Then by Proposition \ref{prop:odd-primitive}, $A$ contains a non-zero odd primitive,
say $x$. Set $B = A/(x)$. Then we have a short exact sequence, $\wedge(\Bbbk x) \to A \to B$, of
Hopf superalgebras, whence $\overline{A} = \overline{B}$, $\dim W^A - 1 = \dim W^B$
by Proposition \ref{prop:exact-seq-com}(3). The induction hypothesis shows 
that $B$ is irreducible.  This, together with the previous lemma
applied to $\operatorname{SSp} B \to G \to \operatorname{SSp}(\wedge(\Bbbk x))$, proves the desired result. 
\end{proof}

\section{Linearly reductive affine supergroup schemes\\ in positive characteristic}\label{sec:Nagata}

\begin{lemma}\label{lem:irredHopf}
An irreducible Hopf superalgebra $H$ is purely even if the Lie superalgebra
$P(H)$ of primitives in $H$ is purely even.
\end{lemma}

\begin{proof}
Suppose that $H$ is a Hopf superalgebra. 
One sees from the wedge product construction given in \cite[Sect.~9.1]{Sw} that
the subcoalgebras which appear in the coradical filtration of $H$ are super-subcoalgebras.
If $H$ is irreducible, then we see as proving \cite[Theorem 9.2.2]{Sw} that 
the associated graded coalgebra $H^{\operatorname{gr}}$ is an irreducible Hopf superalgebra. 
Therefore, by replacing 
$H$ with $H^{\operatorname{gr}}$, we may suppose that
$H$ is \emph{strictly graded} in the sense of \cite[Definition on Page 232]{Sw} 
that $H(0) = \Bbbk$, $H(1) = P(H)$. By \cite[Lemma 11.2.1]{Sw}, this implies that for every $n>1$, 
the $\mathbb{Z}_2$-graded linear map  
$$H(n) \to P(H)^{\otimes n}$$
obtained by composing 
the $n-1$ times iterated coproduct $\Delta_{n-1}|_{H(n)} : H(n) \to H^{\otimes n}$ restricted to $H(n)$, 
with the projection $H^{\otimes n} \to H(1)^{\otimes n}= P(H)^{\otimes n}$, is an injection. Therefore, if
$P(H)$ is even, then
every $H(n)$, and hence the whole $H$ are purely even. 
\end{proof}

The lemma just proven enables us to extend \cite[Theorem 0.1]{M2} to the super situation, as follows.

\begin{corollary}\label{cor:finiteNagata}
For an irreducible Hopf superalgebra $H$ of finite dimension in positive characteristic $p >2$,
the following are equivalent:
\begin{itemize}
\item[(i)] $H$ is semisimple as an algebra;
\item[(ii)] $H$ is purely even, and commutative semisimple as an algebra;
\item[(iii)] the base extension $H \otimes \overline{\Bbbk}$ to the algebraic closure $\overline{\Bbbk}$ of
$\Bbbk$ is isomorphic to the function algebra $\overline{\Bbbk}^G$ of some finite $p$-group $G$; 
\item[(iv)] the Lie superalgebra $P(H)$ is purely even, and is a torus in the sense that 
$P(H)$ is abelian as a Lie algebra, and every element in $P(H)$ generates a semisimple subalgebra
in $H$. 
\end{itemize}
\end{corollary}
\begin{proof}
Obviously, (ii) $\Rightarrow$ (i). It is easy to see 
(ii) $\Leftrightarrow$ (iii). 
We see (iii) $\Rightarrow$ (iv), (iv) $\Rightarrow$ (ii), from the same implications of
\cite[Theorem 0.1]{M2}, combined with Lemma \ref{lem:irredHopf}.
 
To see (i) $\Rightarrow$ (ii), it suffices, by Lemma \ref{lem:irredHopf},
to prove that if $H$ is semisimple, it cannot contain a non-zero odd primitive. 
To the contrary, suppose that
$H$ contains such a primitive, say $x$. Let $K$ be the Hopf super-subalgebra of $H$ generated by
$x$. Then, $x^2$ is an even primitive in $K$, and the quotient Hopf superalgebra $K/(x^2)$
divided by the Hopf super-ideal generated by $x^2$ is isomorphic to $\wedge(\Bbbk x)$.

As a general fact we have that given a finite-dimensional semisimple Hopf superalgebra, every Hopf super-subalgebra, as well as every quotient Hopf superalgebra, is semisimple. 
This follows, for example, from 
\eqref{Rad}, and from the corresponding result in the non-super situation (see \cite[Corollary 3.2.3]{Mon}),
with the bosonization technique \cite[Sect.~10]{MZ} applied. 
But, this general fact implies that the Hopf superalgebra $\wedge(\Bbbk x)$ above is semisimple,
which is absurd.  
\end{proof}

Recall from Definition \ref{def:unipotent-reductive}(2) the definition of 
linearly reductive affine supergroup schemes.

\begin{lemma}\label{lem:linearly-reductive}
A closed normal super-subgroup of a linearly reductive affine supergroup scheme is 
linearly reductive. 
\end{lemma}
\begin{proof}
Just as in the non-super situation, it is known a Hopf superalgebra, say $A$, is cosemisimple
if and only if the purely even, trivial $A$-comodule $\Bbbk$ is injective. By \cite[Theorem 5.9(2)]{M},
given a conormal quotient $A \to D$ of a super-commutative Hopf superalgebra $A$, $A$ is injective
as a left and right $D$-comodule. It follows that if $A$ is cosemisimple, then $D$ is, too. This
is translated into the statement of the corollary. 
\end{proof}

Weissauer \cite[Theorem 6]{W} determined the form of
linearly reductive algebraic affine supergroup schemes over an algebraically closed field 
of characteristic zero. According to his result, those supergroup schemes are rather restricted.
In positive characteristic they are even more so, as is seen from the following.  

\begin{theorem}\label{thm:Nagata}
Assume $\operatorname{char} \Bbbk > 2$, and let $G = \operatorname{SSp} A$ be a linearly reductive affine
supergroup scheme. Then $G$ is necessarily purely even. It follows by Nagata's Theorem
 (see \cite[Chap.~IV, Sect.~3, 3.6]{DG}) that if $G$ is algebraic and connected in addition, then it is of multiplicative 
type in the sense that  
the Hopf algebra $A$ is spanned by grouplikes after the base extension to the algebraic closure
$\overline{\Bbbk}$ of $\Bbbk$; in particular, $G$ is abelian. 
\end{theorem}

\begin{proof}
We prove only the first assertion, which implies immediately the second. 
Let $p = \operatorname{char} \Bbbk \, (>2)$.  We may suppose that $A$ is affine, since it is directed union
of affine Hopf super-subalgebras, which are necessarily cosemisimple. 
We may also suppose by base extension that $\Bbbk$ is algebraically closed. 
Then the Frobenius morphism gives rise to a short exact sequence 
\begin{equation*}
G_1 \to G \to G^{(1)}
\end{equation*}
of affine supergroup schemes. Since $G^{(1)}$, represented by 
$ A^{(1)} := \{a^{p} \mid a \in A\} $, is purely even, it suffices, in virtue of Proposition 
\ref{prop:exact-seq-com}(3),
to show that $G_1$ is purely even. 
Note that $G_1$ is represented by the Hopf superalgebra $A_1 := A/I$,
where $I$ is the Hopf super-ideal of $A$ generated by all $a^{p}$, where $a \in A_0^+$, and so that 
an isomorphism $A \overset{\simeq}{\longrightarrow} \overline{A}\otimes \wedge(W^A)$ 
such as given in Proposition \ref{prop:M2} induces a counit-preserving superalgebra isomorphism 
$A_1 \overset{\simeq}{\longrightarrow} \overline{A}/\overline{I} \otimes \wedge(W^A)$, 
where $\overline{I}$ is the ideal of $\overline{A}$ generated by all $a^{p}$, where $a \in \overline{A}^+$.
The affinity of $A$ implies that $A_1$ is finite-dimensional and local, having $A_1^+$
as a unique maximal ideal. Hence $A_1^*$ is irreducible. 
By Lemma \ref{lem:linearly-reductive}, $A_1$ is cosemisimple, whence $A_1^*$
is semisimple. 
Therefore we can apply Corollary \ref{cor:finiteNagata} to $A_1^*$ to conclude that $G_1$ is purely even, as desired.
\end{proof}

\section*{Acknowledgments}
The work was supported by
Grant-in-Aid for Scientific Research (C) 23540039, Japan Society of the Promotion of Science. 
The author thanks Alexandr Zubkov for his valuable comments on an earlier
version of this paper. 
The author also  thanks the referees for their very helpful comments and useful suggestions.

\end{document}